\documentclass[12pt]{article} 
\usepackage{amsmath}
\usepackage{amsfonts,amssymb}
\usepackage{fullpage}
\usepackage[pdftex]{graphicx}
\usepackage[usenames,dvipsnames]{color}
\usepackage{epstopdf}
\usepackage{epsfig}
\usepackage{subfig}
\usepackage{amsmath}
\usepackage{amsfonts,amssymb}
\usepackage{amssymb}
\usepackage{latexsym}
\usepackage{url}

\newcommand{\convas}{\stackrel{a.s.}{\longrightarrow}}
\newcommand{\convD}{\stackrel{D}{\longrightarrow}}
\newcommand{\convp}{\stackrel{p}{\longrightarrow}}
\newcommand{\sge}{\stackrel{st}{\ge}}
\newcommand{\eqD}{\stackrel{D}{=}}

\newcommand{\hfigwidth}{0.45\textwidth}
\newcommand{\figwidth}{\textwidth}

\newenvironment{proof}[1][Proof]
              {\par \normalfont
              \trivlist
             \item[
               \hspace{12pt}               \itshape #1{.}]\ignorespaces
        }{\hfill$\Box$ \endtrivlist}
\newtheorem{theorem}{Theorem}[section]

\newtheorem{lem}{Lemma}[section]
\newtheorem{cor}{Corollary}[section]

\numberwithin{equation}{section}

\begin{document}


\title{An epidemic in a dynamic population with importation of infectives}


\author{Frank Ball$^{1}$, Tom Britton$^{2}$ and Pieter Trapman$^{2}$}
\footnotetext[1]{University of Nottingham, School of Mathematical Sciences, University Park, Nottingham NG7 2RD, UK,}
\footnotetext[2]{Stockholm University, Department of Mathematics, 106 91 Stockholm, Sweden.}
\date{\today}
\maketitle

\begin{abstract}

Consider a large uniformly mixing dynamic population, which has constant birth rate and exponentially distributed lifetimes, with mean population size $n$. A Markovian SIR (susceptible $\to$ infective $\to$ recovered) infectious disease,  having importation of infectives, taking place in this population is analysed. The main situation treated is where $n\to\infty$, keeping the basic reproduction number $R_0$ as well as the importation rate of infectives fixed, but assuming that the quotient of the average infectious period and the average lifetime tends to 0 faster than $1/\log n$. It is shown that, as $ n \to \infty$, the behaviour of the  3-dimensional process describing the evolution of the fraction of the population that are susceptible, infective and recovered, is encapsulated
in a 1-dimensional regenerative process $S=\{ S(t);t\ge 0\}$ describing the limiting fraction of the population that are susceptible.  The process $S$ grows deterministically, except at one random time point per regenerative cycle, where it jumps down by a size that is completely determined by the waiting time since the previous jump.
Properties of the process $S$, including the jump size and stationary distributions, are determined.

\end{abstract}

\section{Introduction} \label{Sec_intro}
The mathematical theory for the spread of infectious diseases has a long history and is by now quite rich (e.g.,~\cite{DHB13}). One of the more common type of disease models is called SIR (susceptible $\to$ infective $\to$ recovered) meaning that individuals are at first Susceptible. If infected (by someone) they immediately become Infectious (being able to spread the disease onwards). After some time an infectious individual Recovers, which also means that the individual is immune to further infection from the disease. Such models were originally studied for populations assuming homogeneous mixing, but during the last few decades considerable effort has been put into analysing epidemic models in communities which are not homogeneously mixing but instead may be described using some type of social structure, such as a community of households (e.g.~\cite{BMS-T97}) or a random network describing possible contacts (e.g.~\cite{Newman02}). The vast majority of papers devoted to these type of problems assume a fixed community and community structure.

In the current paper we treat the situation where the population is dynamic in the sense that people die and new individuals are born, or more precisely immigrate into the population. Further, we assume that there is also importation of infectious individuals (randomly in time according to a homogeneous Poisson process), implying that the disease never vanishes forever. In order to facilitate analytical progress we consider only the case of a homogeneously mixing community, which in network terminology corresponds to treating the complete network.

Models for recurrent epidemics go back to the deterministic formulations of~\cite{Hamer} and~\cite{Soper}.  A stochastic treatment was given first in the pioneering work of~\cite{Bart56}, who considered an SIR model with importation of both susceptibles and infectives, but without disease-unrelated deaths.  An alternative model, with disease-unrelated deaths but no importation of infectives, has been studied extensively (e.g.~\cite{Nasell99} and the references therein).  Interest often centres on the time to extinction of infection and the closely-related problem of the critical community size for an infection to persist in a population.

We consider a Markovian SIR epidemic with demography and importation of infectives, in which infectious individuals infect new individuals at constant rate and the infectious period is exponentially distributed. We study limit properties of the epidemic when the average population size $n$ tends to infinity. Our focus lies on the case where the limit is taken keeping the basic reproduction number $R_0$ (i.e.~the average number of susceptibles infected by a single infective in an otherwise fully susceptible population of size $n$) and the immigration rate of infectives fixed, whereas the quotient of the average infectious period and the average lifetime tends to 0 faster than $1/\log n$. For many infectious diseases this quotient typically lies between $10^{-4}$ and $10^{-3}$, hence supporting this asymptotic regime, but in the discussion we treat other asymptotic regimes briefly.

Under the above asymptotic regime, all epidemic outbreaks are short, having duration that tends to $0$ in probability as $n \to \infty$.  Further, as $n \to \infty$, epidemic outbreaks are either minor, having size of order $o_p(n)$, or major, having size of exact order $\Theta_p(n)$.  It follows that, as $n \to \infty$,
the behaviour of the three-dimensional process describing the evolution of the fraction of the population that are susceptible, infective and recovered, is encapsulated
in a one-dimensional regenerative process $S=\{ S(t);t\ge 0\}$, describing the limiting fraction of the population that are susceptible.
During each cycle, the process $S$ makes one down jump, corresponding to the occurrence of a major outbreak, and except for this increases deterministically, as minor outbreaks have no effect on $\bar S^{(n)}$ in the limit as $n \to \infty$.
(Here, $\bar S^{(n)}=\{\bar S^{(n)}(t):t \ge 0\}$, where, for $t \ge 0$, $\bar S^{(n)}(t)=n^{-1}S^{(n)}(t)$ with $S^{(n)}(t)$ being the number of susceptible individuals in the population at time $t$.) Note that $\bar S^{(n)}$ does not converge weakly to $S$ in the Skorohod topology since the sample paths of $S$ are almost surely discontinuous but those of $\bar S^{(n)}$ almost surely contain only jumps of size $n^{-1}$, so are close to being continuous.  Thus to obtain rigorous convergence results, we consider two processes, $\bar S^{(n)}_-$ and $\bar S^{(n)}_+$, which coincide with
$\bar S^{(n)}$, except during major outbreaks during which they sandwich $\bar S^{(n)}$, and prove that both $\bar S^{(n)}_-$ and $\bar S^{(n)}_+$ converge weakly to $S$ in the Skorohod topology (Theorem~\ref{wconv}).  It then follows that certain functionals of $\bar S^{(n)}$ converge weakly to corresponding functionals of $S$ (Corollary~\ref{convfunct}).


The paper is structured as follows. In Section~\ref{modelmain}, we define the model and the limiting regenerative process, give an intuitive explanation of why $S$ approximates $\bar S^{(n)}$ for large $n$ and present the main convergence results. In Section~\ref{S-properties}, we derive some properties of the limiting regenerative process: the jump size distribution, the associated renewal time distribution and the stationary distribution. In Section~\ref{illustrations}, we present simulations supporting the convergence result and illustrating various features of the limiting process. In Section~\ref{proofs}, we prove the main results. We end in Section~\ref{disc} with a Discussion summarising our results and also exploring briefly additional questions, such as other asymptotic regimes.

\section{The epidemic model and main results}
\label{modelmain}

\subsection{The Markovian SIR epidemic with demography and importation of infectives}
\label{model}

We now define the Markovian SIR epidemic with demography and importation of infectives (SIR-D-I). We consider the process to be indexed by a target population size $n$, which we assume is a strictly positive constant. 
The population model is an immigration-death process with constant immigration rate and linear death rate. For $t \ge 0$, let $N^{(n)}(t)$ denote the population size at time $t$. Then $N^{(n)}(t)$ increases at constant rate $\mu n$ and decreases at rate $\mu N^{(n)}(t)$. The population size  hence fluctuates around $n$, which is assumed to be large.

The Markovian SIR-epidemic on this population is defined as follows. For $t \ge 0$, let $S^{(n)}(t),\ I^{(n)}(t)$ and $R^{(n)}(t)$ denote the number of susceptibles, infectives and recovered, respectively, at time $t$, so $S^{(n)}(t)+I^{(n)}(t)+R^{(n)}(t)=N^{(n)}(t)$. We assume that $I^{(n)}(0)=0$ and that $\bar S^{(n)}(0) \to s_0$ as $n \to \infty$, where $s_0 \in (0,1]$ is constant.  (The value of $R^{(n)}(0)$ has no effect on the ensuing epidemic.) . A fraction $\kappa_n$ of all births (i.e.~immigrants)  are infectives and the remaining births are all susceptibles, so births of infectives occur at rate $\mu n\kappa_n$ and births of susceptibles occur at rate $\mu n(1-\kappa_n)$. While infectious, any given infective infects any given susceptible at rate $n^{-1} \lambda_n$, independently between each distinct pair of individuals.  Thus, approximately, each infective makes infectious contacts at the points of a homogeneous Poisson process having rate $\lambda_n$, with contacts being with individuals chosen independently and uniformly from the whole population; a contact with a susceptible individual results in that individual becoming infected, while a contact with an infectious or removed individual has no effect.
Each infectious individual recovers and becomes immune at rate $\gamma_n$, implying that the infectious period is exponentially distributed with rate parameter $\gamma_n$.

More formally, the process $\left\{\left(S^{(n)}(t),I^{(n)}(t), R^{(n)}(t)\right):t \ge 0\right\}$ is a continuous-time Markov chain, with state space $\mathbb{Z}_+^3$ and transition intensities given by
\begin{align*}
q^{(n)}_{(s,i,r),(s+1,i,r)}&=(1-\kappa_n)n\mu,\\
q^{(n)}_{(s,i,r),(s,i+1,r)}&=\kappa_n n\mu,\\
q^{(n)}_{(s,i,r),(s-1,i,r)}&=\mu s,\\
q^{(n)}_{(s,i,r),(s,i-1,r)}&=\mu i,\\
q^{(n)}_{(s,i,r),(s,i,r-1)}&=\mu r,\\
q^{(n)}_{(s,i,r),(s-1,i+1,r)}&=n^{-1}\lambda_n s i,\\
q^{(n)}_{(s,i,r),(s,i-1,r+1)}&=\gamma_n i,
\end{align*}
corresponding to birth of a susceptible, birth of an infective, death of a susceptible, death of an infective, death of a recovered, infection of a susceptible and recovery of an infective, respectively.

We study specifically the case where the average population size $n$ tends to infinity in such a way that
\begin{itemize}
\item[(a)] the total importation rate  $\mu n\kappa_n$ of infectives tends to a strictly positive constant $\mu\kappa$, so $\kappa_n n\to \kappa$ as $n\to\infty$; and
\item[(b)] the infection and recovery rates satisfy $\lambda_n/\gamma_n\to R_0>1$ and \newline $\lambda_n/\log n\to \infty$ as $n\to\infty$.
\end{itemize}
For ease of exposition, we assume that $n$ is an integer, so sequences of epidemic processes are indexed by the natural numbers.  However, all of the results of the paper are easily generalised to the case of a family of epidemic processes indexed by the positive real numbers. 

To conclude, the parameters of the model are: $n$, the average population size; $\mu$, where $1/\mu$ is the average lifetime and $\mu n$ is the population birth rate; $\lambda_n$, the infection rate; $\gamma_n$, where $1/\gamma_n$ is the average length of the infectious period; and $\kappa_n$, the fraction of births which are infectious, so $\mu n\kappa_n$ is the birth (or importation) rate of infectives.

\subsection{The limiting process $S$}\label{S(t)}

Let $\bar S^{(n)}=\left\{\bar S^{(n)}(t):t \ge 0\right\}$, where $\bar S^{(n)}(t)=n^{-1}S^{(n)}(t)$ is the ``fraction" of the population that is susceptible at time $t$. The process $S=\{ S(t);t\ge 0\}$
can be viewed as the limit of $\bar S^{(n)}$ as $n \to \infty$ under the above asymptotic regime. It
is a Markovian regenerative process (e.g.~\cite{Asmussen87}, Chapter V), with renewals occurring whenever $S(t)=1/R_0$. Between each renewal $S(t)$ increases deterministically according to the differential equation
\begin{equation}
\label{Sdet}
S'(t)= \mu(1-S(t)),
\end{equation}
except for one down jump (from above $1/R_0$ to below $1/R_0$). This implies that
\begin{equation}
\label{Sbeforejump}
S(u)=1-(1-1/R_0)e^{-\mu u}
\end{equation}
before the jump (if $u$ denotes the time from the last renewal). The random time $T$ from a renewal to the jump has distribution specified by
\[
{\rm P}(T\le t)=1-\exp\left[-\mu\kappa \int_0^t\left(1-\frac{1}{R_0S(u)}\right)\,\mathrm{d}u\right]\quad(t \ge 0),
\]
with $S(u)$ given by~\eqref{Sbeforejump}, so
\begin{equation}
\label{FT}
{\rm P}(T\le t)= 1-{\rm e}^{-\mu \kappa t}\left(R_0{\rm e}^{\mu t}-R_0+1\right)^{\frac{\kappa}{R_0}}\quad(t \ge 0).
\end{equation}
The size of the jump is specified by the value $S(T-)$ of the process just prior to the jump.   More precisely, 
$S(T)=S(T-)(1-\tau(S(T-)))$, where for $s>R_0^{-1}$, $\tau(s)$ is the unique strictly positive solution to the equation (cf.~\cite{DHB13}, equation (3.15))
\begin{equation}
\label{tau}
1-\tau={\rm e}^{-R_0 s\tau}.
\end{equation}
In epidemic theory $\tau(s)$ is known as the relative fraction infected among the initially susceptible of an SIR epidemic outbreak in which a fraction $s$ are initially susceptible and the rest immune. Hence, the size of the down jump is $S(T-)\tau(S(T-))$. After the down jump, $S(t)$ increases deterministically according to the same differential equation~\eqref{Sdet} until the next renewal point, so
\[
S(T+t)=1-(1-S(T)){\rm e}^{-\mu t},\ 0\le t\le \mu^{-1}\log [(1-S(T))/(1-1/R_0)]
\]
and the inter-renewal time is $T + \mu^{-1}\log [(1-S(T))/(1-1/R_0)]$. Illustrations of $S$ are given in Section~\ref{illustrations}.

\subsection{Main results and heuristics}
\label{subsecmain}

We first explain heuristically why $S$ can be viewed as the limit of $\bar S^{(n)}$ as $n \to \infty$ under that asymptotic regime described in Section~\ref{model}.   Suppose that $n$ is large. Then when no infective is present, all that happens is that individuals die and new ones are born at approximately the same rate $\mu n$. Recovered (immune) individuals that die are replaced by susceptible individuals, so the fraction of susceptibles increases at rate $\mu (1-\bar{S}^{(n)}(t))$ which explains the deterministic growth rate of $S$.

After an exponentially distributed holding time, with rate parameter $\mu n\kappa_n\approx \mu \kappa$, an infective is born into the community. If the fraction susceptible $\bar{S}^{(n)}(t)$ is below $1/R_0$, then the effective reproduction number $R_e=R_0\bar{S}^{(n)}(t)$ is strictly less than one, implying that, with probability tending to one as $n \to \infty$, a large outbreak will not occur, so $\bar{S}^{(n)}(t)$ continues to grow approximately deterministically. If $\bar{S}^{(n)}(t)>1/R_0$ when a new born infective enters the community, then with approximate probability  $1-1/(R_0 \bar{S}^{(n)}(t))$ that infective gives rise to a major outbreak that infects order $\Theta(n)$ susceptibles (cf.~\cite{DHB13}, pages 53 and 376), otherwise only a minor outbreak, which infects order $o(n)$ susceptibles, occurs and $\bar{S}^{(n)}(t)$ continues to grow approximately deterministically.
This explains the distribution for $T$, the time from a renewal until a down jump in $S$, which has time varying intensity given by $\mu\kappa$ multiplied by the limiting major outbreak probability (cf.~\cite{Bart56}).

If a major outbreak takes place, the size of the outbreak among the susceptibles is given approximately by $\tau(S(T-)){S}^{(n)}(T-)$ where $S(T-)$ denotes the limiting (as $n\to \infty$) fraction susceptible just prior to the outbreak and $\tau (s)$ is defined above (cf.~\cite{DHB13}, page 60). The duration of such a major outbreak is of order $\Theta(\log n/\lambda_n)$ (cf.~\cite{Barbour75}) which tends to $0$ by assumption. Thus, if there is a major outbreak it happens momentarily and, in the limit as $n \to \infty$, the fraction susceptible after the outbreak, $S(T)$, satisfies $S(T)=S(T-)(1-\tau (S(T-))$.

Although the above heuristic argument makes it plausible that the normalised susceptible process $\bar{S}^{(n)}$ converges to the regenerative process $S$, there are two complicating factors in making the argument fully rigorous.
First, as explained in Section~\ref{Sec_intro}, it is not true that $\bar S^{(n)} \Rightarrow S$ as $ n \to \infty$, where $\Rightarrow$ denotes weak convergence in the space $D[0,\infty)$ of right-continuous functions $f:[0,\infty) \to \mathbb{R}$ having limits from the left (i.e.~c\`{a}dl\`{a}g functions), endowed with the Skorohod metric (e.g.~\cite{EK86}, Chapter 3). As explained also in Section~\ref{Sec_intro}, we overcome this problem by considering two processes, $\bar S^{(n)}_-$ and $\bar S^{(n)}_+$, which coincide with $\bar S^{(n)}$ except during major outbreaks, when they sandwich $\bar S^{(n)}$, and show that $\bar S^{(n)}_-\Rightarrow S$ and $\bar S^{(n)}_+ \Rightarrow S(\cdot)$ as $n \to \infty$; see Theorem~\ref{wconv}.
The second complicating factor is that the results referred to above concerning the probability, size and duration of a major outbreak are for an epidemic in a static
population, whereas our population is dynamic.  The results carry over to our setting because, in the limit as $n \to \infty$, the time scale of an epidemic outbreak is infinitely faster than that of demographic change, but proofs need to be adapted accordingly.

Before stating our main theorem, some more notation is required.
Recall that $I^{(n)}(t)$ is the number of infectives at time $t$ in the SIR-D-I epidemic with average population size $n$ and that we consider epidemics with no infective at time $0$, i.e.~with $I^{(n)}(0)=0$.  Let $t^{(n)}_0=u^{(n)}_0=0$.  For $k=1,2,\cdots$, let $t^{(n)}_k=\inf\{t\ge u^{(n)}_{k-1}:I^{(n)}(t) \ge \log n\}$ and $u^{(n)}_k=\inf\{t\ge t^{(n)}_{k}:I^{(n)}(t)=0\}$.  Thus, provided $n$ is sufficiently large, the $k$th major outbreak starts at approximately time $t^{(n)}_k$ and ends at time  $u^{(n)}_k$.  (The choice of $\log n$ to delineate major outbreaks is essentially arbitrary.  Our proofs work equally well if $\log n$ is replaced by any function $g(n)$ which satisfies
$g(n) \to \infty$ and $n^{-\frac{1}{2}} g(n) \to 0$ as $n \to \infty$.)
For $t \ge 0$, let
\[
\bar S^{(n)}_-(t)  = \left\{
  \begin{array}{l l}
   \bar S^{(n)}(t) & \quad \text{if } t \notin [t^{(n)}_i,u^{(n)}_i) \text{ for some }i,\\
   \min_{t^{(n)}_i \le t' \le u^{(n)}_i} \bar S^{(n)}(t') & \quad \text{if } t \in [t^{(n)}_i,u^{(n)}_i),i=1,2,\cdots,
  \end{array} \right.
\]
and
\[
\bar S^{(n)}_+(t)  = \left\{
  \begin{array}{l l}
   \bar S^{(n)}(t) & \quad \text{if } t \notin [t^{(n)}_i,u^{(n)}_i) \text{ for some }i,\\
   \max_{t^{(n)}_i \le t' \le u^{(n)}_i} \bar S^{(n)}(t') & \quad \text{if } t \in [t^{(n)}_i,u^{(n)}_i),i=1,2,\cdots.
  \end{array} \right.
\]

The following theorem is proved in Section~\ref{proof}.

\begin{theorem}
\label{wconv}
Suppose that $\lim_{n \to \infty} \bar S^{(n)}(0)=s_0$.  Then, as $n \to \infty$,
\[
\bar S^{(n)}_- \Rightarrow  S \quad \mbox{and} \quad  \bar S^{(n)}_+ \Rightarrow  S,
\]
where $S(0)=s_0$.
\end{theorem}

An immediate consequence of Theorem~\ref{wconv} is that suitable functionals of $\bar S^{(n)}$ converge weakly to corresponding functionals of $S$.  For $g,h \in D[0,\infty)$, let $g \le h$ denote $g(t) \le h(t)$ for all $t\ge 0$.  A functional $H:D[0,\infty) \to \mathbb{R}$ is called monotone if either $Hf\le Hg$ for all $f,g \in D[0,\infty)$ satisfying $f \le g$, or
$Hf\le Hg$ for all $f,g \in D[0,\infty)$ satisfying $g \le f$.  The following corollary, which can clearly be generalised to suitable non-real-valued functionals, follows immediately from Theorem~\ref{wconv} by using the continuous mapping theorem (e.g.~\cite{Bill}).  For $H:D[0,\infty) \to \mathbb{R}$ , let
$C_H=\{f \in D[0,\infty): H \mbox{ is continuous at } f\}$.

\begin{cor}
\label{convfunct}
Suppose that $\lim_{n \to \infty} \bar S^{(n)}(0)=s_0$,
$H:D[0,\infty) \to \mathbb{R}$ is monotone and ${\rm P}(S \in C_H)=1$. Then
\[
HS^{(n)} \convD HS \quad \text{as} \quad n \to \infty,
\]
where $S(0)=s_0$.
\end{cor}

One functional which satisfies the conditions of Corollary~\ref{convfunct} is the first passage time functional $H_a$, defined for given $a \in (0,1)$ by
\[
H_af  = \left\{
  \begin{array}{l l}
   \inf\{t\ge 0: f(t) \ge a\} & \quad \text{if } f(0) \le a,\\
   \inf\{t\ge 0: f(t) \le a\} & \quad \text{if } f(0) > a .
  \end{array} \right.
\]
The functional $H_a$ is clearly monotone and ${\rm P}\left(S \in C_{H_a}\right)=1$, cf.~\cite{Pollard84}, page 124.

Another functional satisfying the conditions of Corollary~\ref{convfunct} is the occupancy time functional $H_{t^*}^a$,
defined for any given $t^*>0$ and $ a \in (0,1)$ by 
\begin{equation}
\label{occfunct}
H_{t^*}^af=\int_0^{t^*} 1_{\{f(t) \le a\}} \,\mathrm{d}t.
\end{equation}
This functional is again clearly monotone.  The proof that ${\rm P}\left(S \in C_{H_{t^*}^a}\right)=1$ is given at the end of Section~\ref{proof}.


\section{Properties of the limiting process $S$}\label{S-properties}

We now outline some properties of the regenerative process $S$ which can be obtained from renewal and regenerative process theory (e.g.~\cite{Asmussen87}, Chapters IV and V). As described in Section~\ref{S(t)} the stochastic part of the regenerative process is completely specified by the waiting time $T$ until the down jump, but it can be specified equivalently by the jump size $X=S(T-)-S(T)$. Noting that $\tau(S(T-))=(S(T-) -S(T))/S(T-)$), it follows from~\eqref{tau} that
\[
\frac{S(T)}{S(T-)}={\rm e}^{-R_0(S(T-)-S(T))}={\rm e}^{-R_0 X},
\]
whence
\[
S(T-)=\frac{X}{1-e^{-R_0X}}=\frac{X{\rm e}^{R_0X}}{e^{R_0X}-1}\qquad\text{and}\qquad S(T)=\frac{X}{e^{R_0X}-1},
\]
 which can be used to obtain the distribution of the jump size $X$.  The jump size is strictly less than $\tau(1)$, as $S(t)<1$ for all $t \ge 0$. Hence, for $0 < x <\tau(1)$,
\begin{align}
\label{FX}
 F_X(x)& ={\rm P}(X\le x) \nonumber
\\
 &={\rm P}\left(S(T-)\le \frac{x}{1-{\rm e}^{-R_0x}}\right)\nonumber
\\
&= {\rm P}\left( T\le -\mu^{-1}\log\left[ \frac{1-x/(1-{\rm e}^{-R_0x})}{1-1/R_0}\right] \right) \quad(\mbox{using}~\eqref{Sbeforejump})\nonumber
\\
&=1-\left[\frac{R_0\left(1-x-{\rm e}^{-R_0 x}\right)}{\left(R_0-1\right)\left(1-{\rm e}^{-R_0 x}\right)}\right]^\kappa
\left[\frac{\left(R_0-1\right)x}{1-x-{\rm e}^{-R_0 x}}\right]^{\frac{\kappa}{R_0}}.
\end{align}

The lifetime distribution for the renewal process describing successive visits of $S$ to $1/R_0$
may be derived as follows. During a cycle, the regenerative process $S$ starts at $1/R_0$ and grows deterministically, according to~\eqref{Sdet}, until the time $T$ of the down jump. After this down jump it again grows deterministically, according to~\eqref{Sdet}, until it reaches $1/R_0$, when the next renewal occurs. If we change the order of these two parts, the process starts at $S(T)$ and grows deterministically until it reaches $S(T-)$. The lifetime $T^*$ hence equals the time it takes for the deterministic curve defined by~\eqref{Sdet} to travel from $S(T)$ to $S(T-)$  This time equals
$$
T^*=\mu^{-1}\log\left( \frac{1-S(T)}{1-S(T-)}\right) = \mu^{-1}\log\left( \frac{{\rm e}^{R_0X}-1-X}{(1-X){\rm e}^{R_0X}-1}\right) .
$$
This is a monotonic increasing function of $X$, so the renewal time distribution can be obtained numerically using the expression $F_X(x)$ given by~\eqref{FX}.

The stationary distribution of $S$ can be obtained using regenerative process theory (e.g.~\cite{Asmussen87}, Chapter V, Section 3). During a regenerative cycle, the process $S$ traverses $s$ if and only if $s$ lies between $S(T)$ and $S(T-)$. If it does, the density for the time spent there is inversely proportional to the derivative $\mu(1-s)$.
Consequently, if we let $f_{S^*}(s)$ denote the density of the stationary distribution of $S$, we have
\begin{equation}
f_{S^*}(s) =\frac{c}{\mu(1-s)}{\rm P}(s\in [S(T), S(T-)])\quad(1-\tau(1)<s<1),
\end{equation}
where $c$ ($=1/{\rm E}[T^*]$) is the normalizing constant making this a pdf. If $s \in [1/R_0,1)$, then
$s\in [S(T), S(T-)]$ if and only if $T \ge \mu^{-1}\log((1-R_0^{-1})/(1-s)$.  If $s \in (1-\tau(1),1/R_0)$, then $s\in [S(T), S(T-)]$ if and only if $ X \ge g^{-1}(s)$, where $g:(0,\tau(1)) \to (1-\tau(1),1/R_0)$ is defined by $g(x)=x/({\rm e}^{R_0 x}-1)$. It then follows using~\eqref{FT} and~\eqref{FX} that, with $\tilde{s}= g^{-1}(s)$,
\[
f_{S^*}(s)=\left\{
  \begin{array}{l l}
   c \left[\frac{R_0\left(1-\tilde{s}-{\rm e}^{-R_0 \tilde{s}}\right)}{\left(R_0-1\right)\left(1-{\rm e}^{-R_0 }\tilde{s}\right)}\right]^\kappa
\left[\frac{\left(R_0-1\right)\tilde{s}}{1-\tilde{s}-{\rm e}^{-R_0 }\tilde{s}}\right]^{\frac{\kappa}{R_0}}&\quad \text{if }  1-\tau(1) <s<1/R_0,\\
   c\left(\frac{1-s}{R_0-1}\right)^{\kappa\left(1-\frac{1}{R_0}\right)}R_0^\kappa s^{\frac{\kappa}{R_0}} & \quad \text{if }  1/R_0 \le s < 1.
  \end{array} \right.
\]
In the next section the density $f_{S^*}(s)$ is calculated numerically and shown to agree with corresponding empirical values from simulations.

\section{Numerical illustrations}\label{illustrations}

We now present briefly some numerical and simulation results, which  illustrate convergence of the epidemic
process as well as properties of the limiting stationary distribution of the fraction susceptible $S^*$. In Figure \ref{epidsimul} the epidemic is simulated for 100 years in a population of $n=10,000$ individuals. In all figures, $R_0=2$ implying that the effective reproduction number $R_e=R_0\bar S^{(n)}(t)$ is supercritical as soon as the population fraction susceptible exceeds $1/R_0=0.5$. The average lifetime is $1/\mu=75$ years and
$\gamma=50$, so the average length of the infectious period is about 1 week. In the left panels of Figure~\ref{epidsimul},
$\kappa=20$, so the rate at which new infectives enter the population ($\mu\kappa$) equals 1 per 3.75 years, and in the right panels $\kappa=200$, so new infectives enter the population at rate $2\frac{2}{3}$ per year.
The upper panels show the fraction of the population that is susceptible over the $100$ period and the lower panels show the corresponding fraction that is infective.  Observe that when $\kappa=20$ major outbreaks are less frequent but larger than when $\kappa=200$, and that there are appreciably more minor outbreaks when $\kappa=100$. Note also that epidemics are rarer than the importation rate of infectives suggests, for two reasons. First, major outbreaks can occur only when $\bar S^{(n)}(t)>1/R_0=0.5$, and secondly, when $\bar S^{(n)}(t)$ is above this threshold, major outbreaks do not occur each time an infective enters the community. In the lower left panel of Figure \ref{epidsimul} some minor outbreaks caused by importation of infectives can also be seen.
\begin{figure}
\begin{center}
\resizebox{\hfigwidth}{!}{\includegraphics{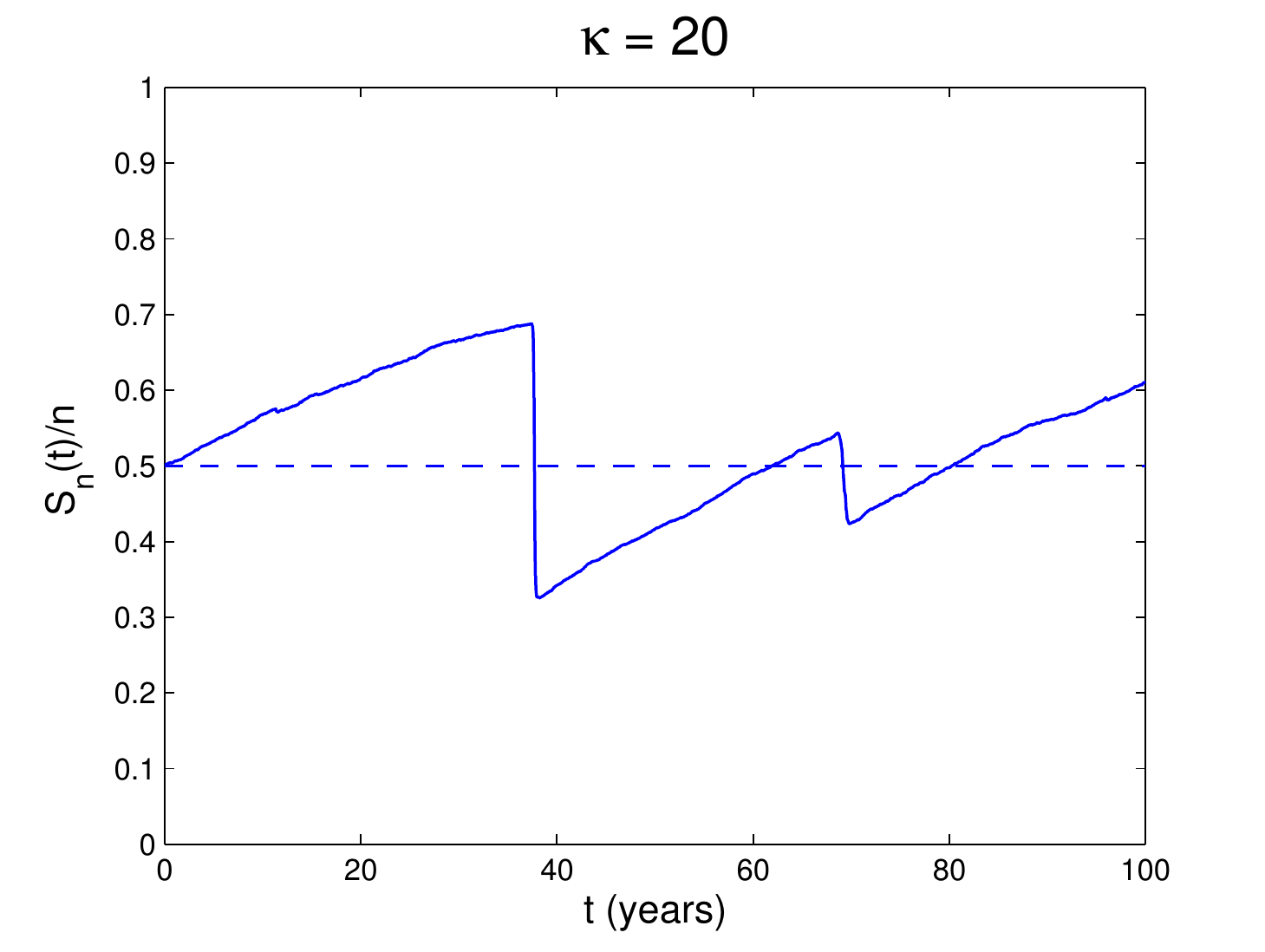}}
\resizebox{\hfigwidth}{!}{\includegraphics{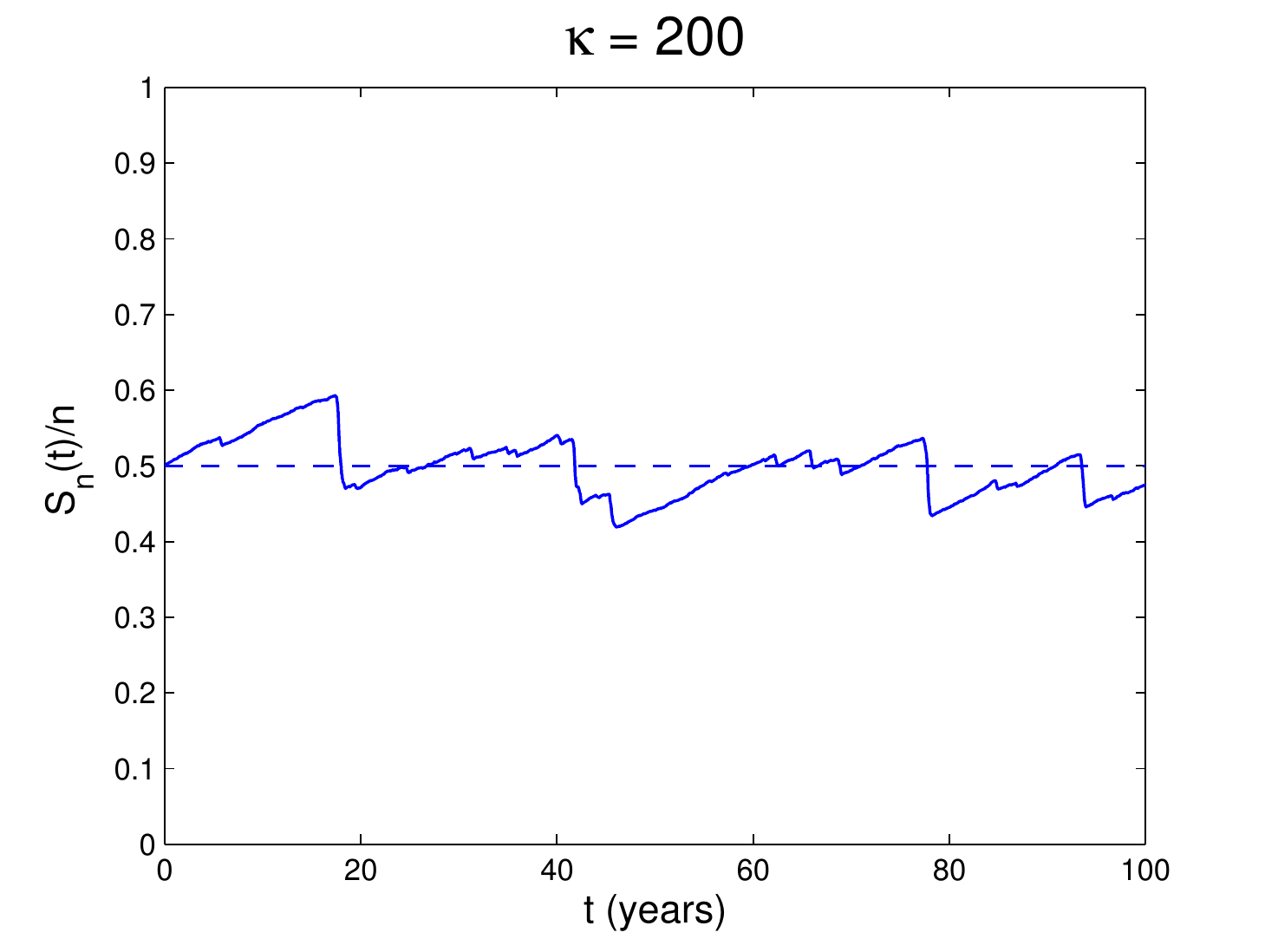}}\\
\resizebox{\hfigwidth}{!}{\includegraphics{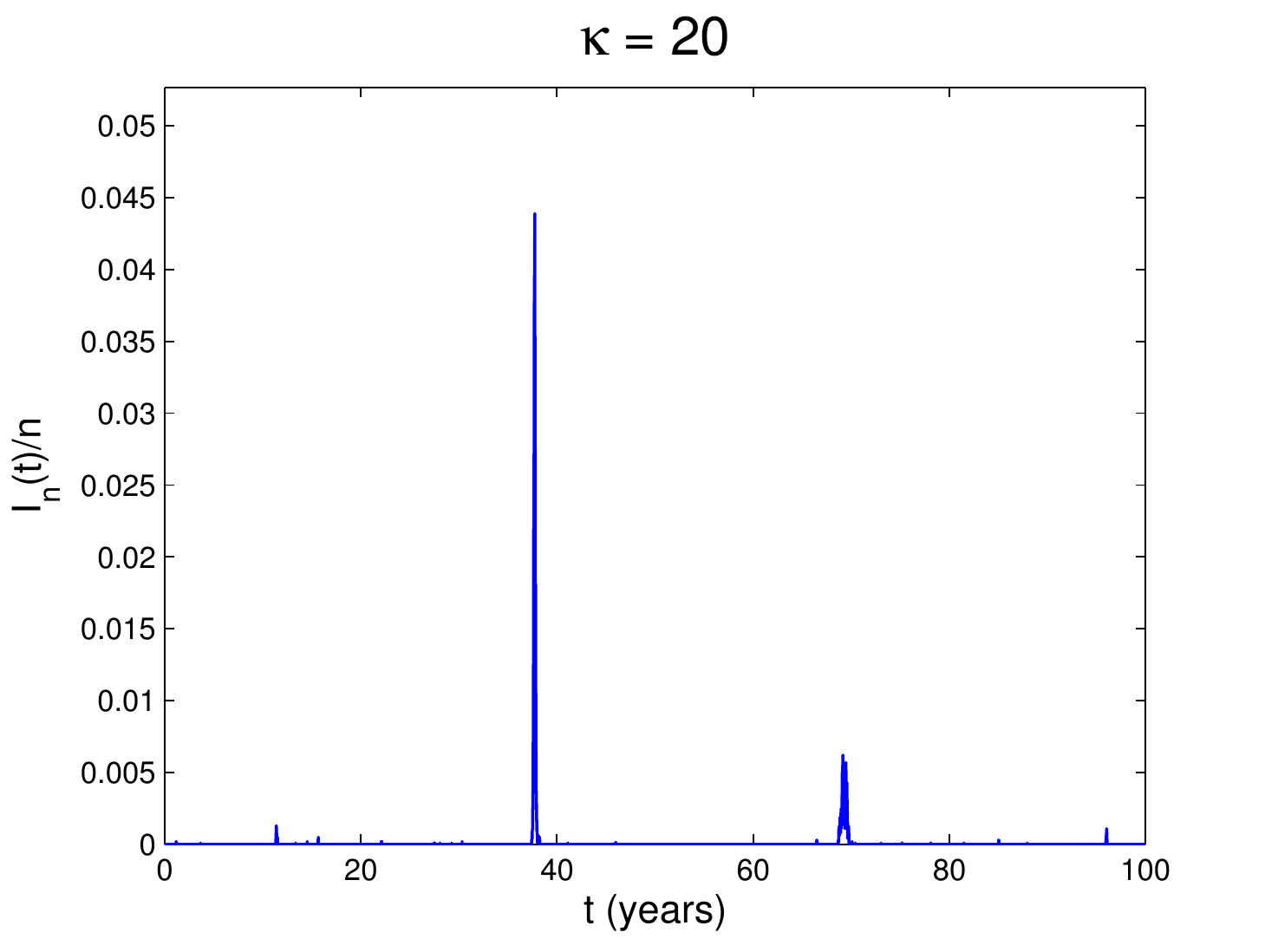}}
\resizebox{\hfigwidth}{!}{\includegraphics{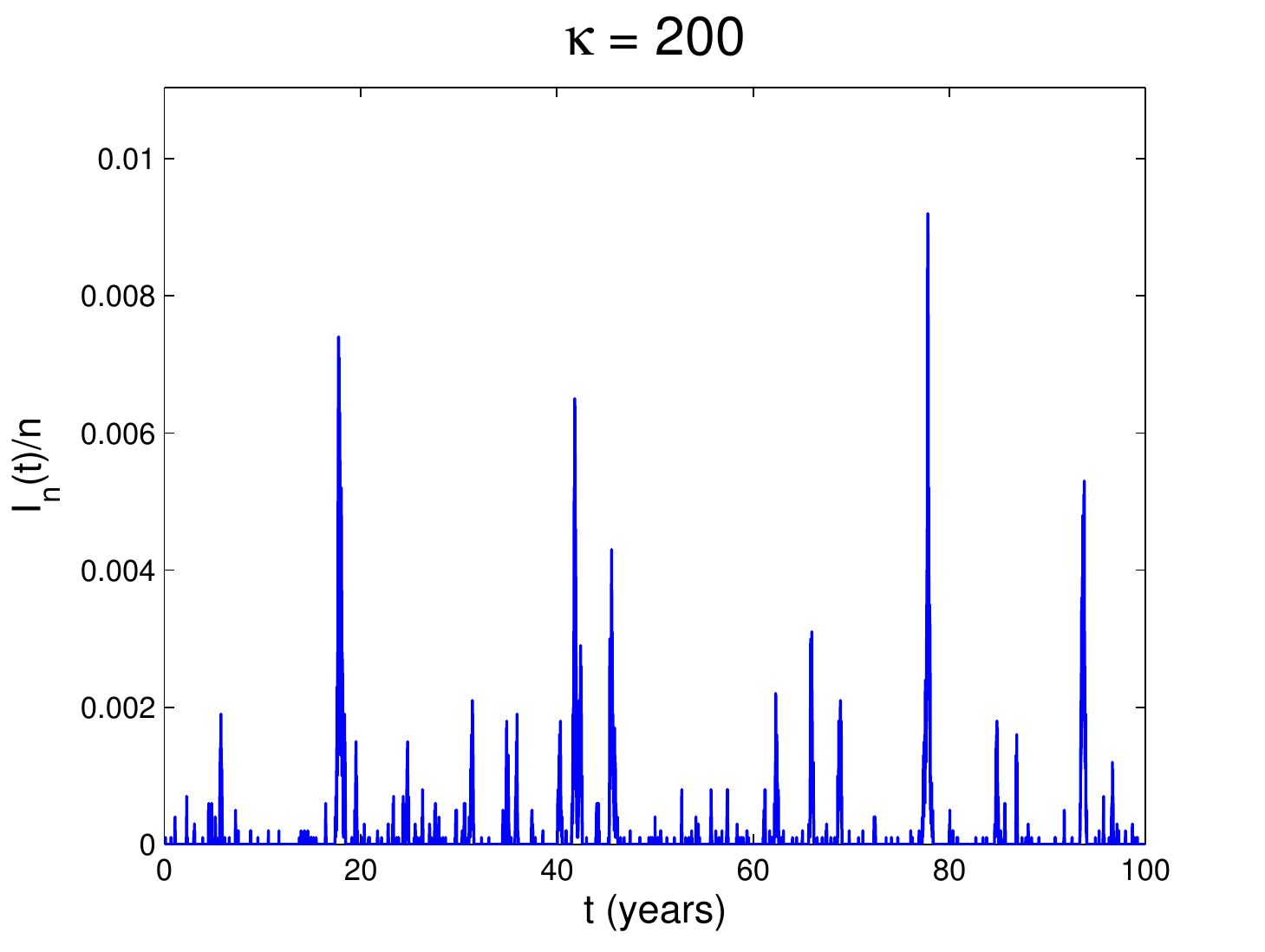}}
\end{center}
\caption{Simulation of the SIR-D-I epidemic with $n=10,000$ individuals, $R_0=2$. In the left panels $\kappa=20$ and $\kappa=200$ in the right panels. The average life length is $1/\mu=75$ years and mean infectious period is $1/\gamma \approx$ 1 week. The fraction of the population susceptible (upper panels) and infective (lower panels) is plotted over a 100 year period in both cases. The dashed line in the upper panels shows the critical fraction susceptible so that the effective reproduction number $R_e=1$.  Note that the scales for the fraction of the population infective are different in the two lower panels; major outbreaks are appreciably larger in the left figure.}
\label{epidsimul}
\end{figure}

In Figure~\ref{limitingsimul} realisations of the corresponding limiting processes are plotted. The same parameter values are used in both figures. The stochastic features of the epidemic and the limiting process are in agreement, suggesting that the limiting behaviour has kicked in when $n=10,000$.  Note that, unlike in Figure~\ref{epidsimul}, there are no near-vertical lines as outbreaks are now instantaneous.
\begin{figure}
\begin{center}
\resizebox{\hfigwidth}{!}{\includegraphics{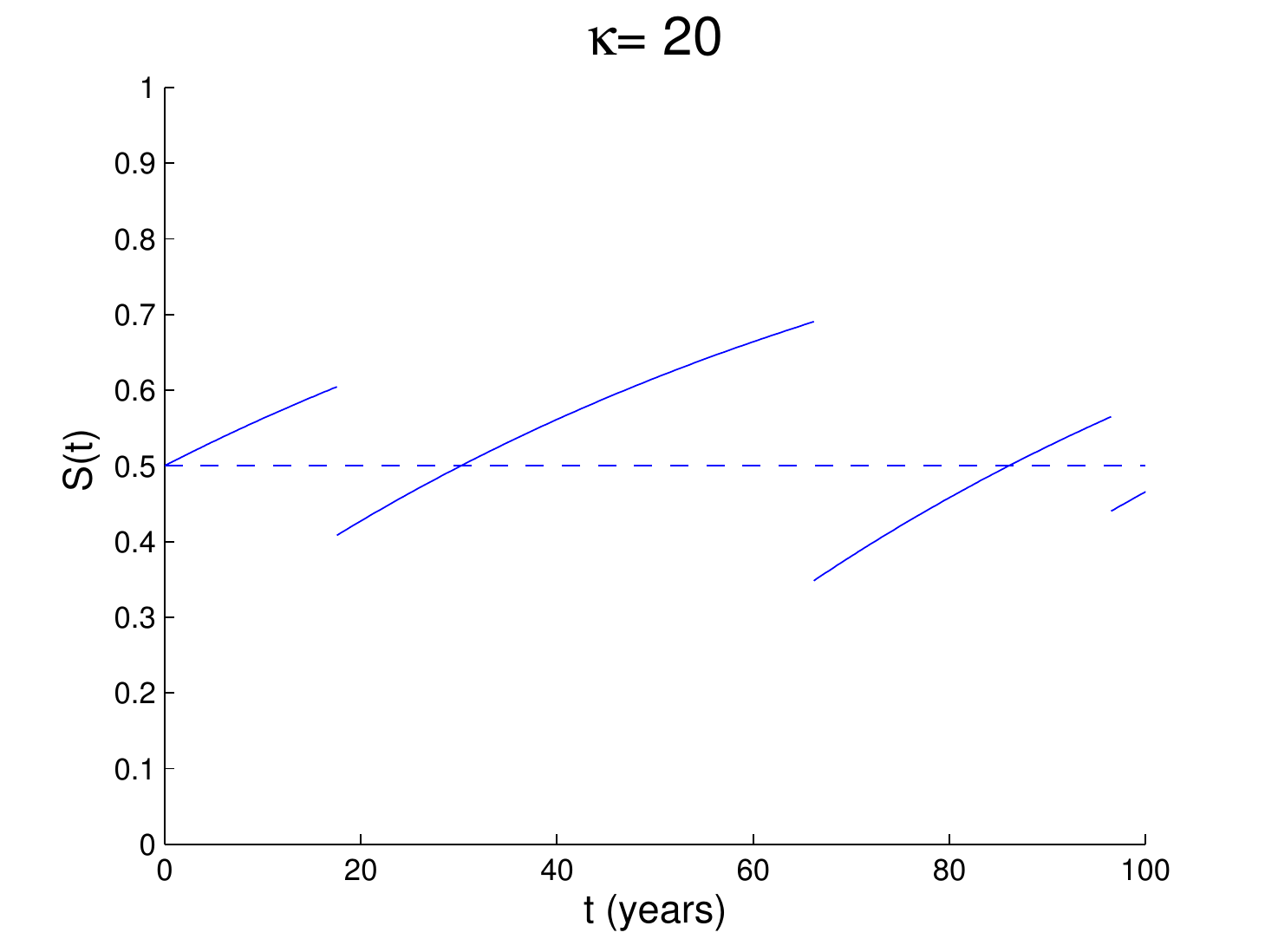}}
\resizebox{\hfigwidth}{!}{\includegraphics{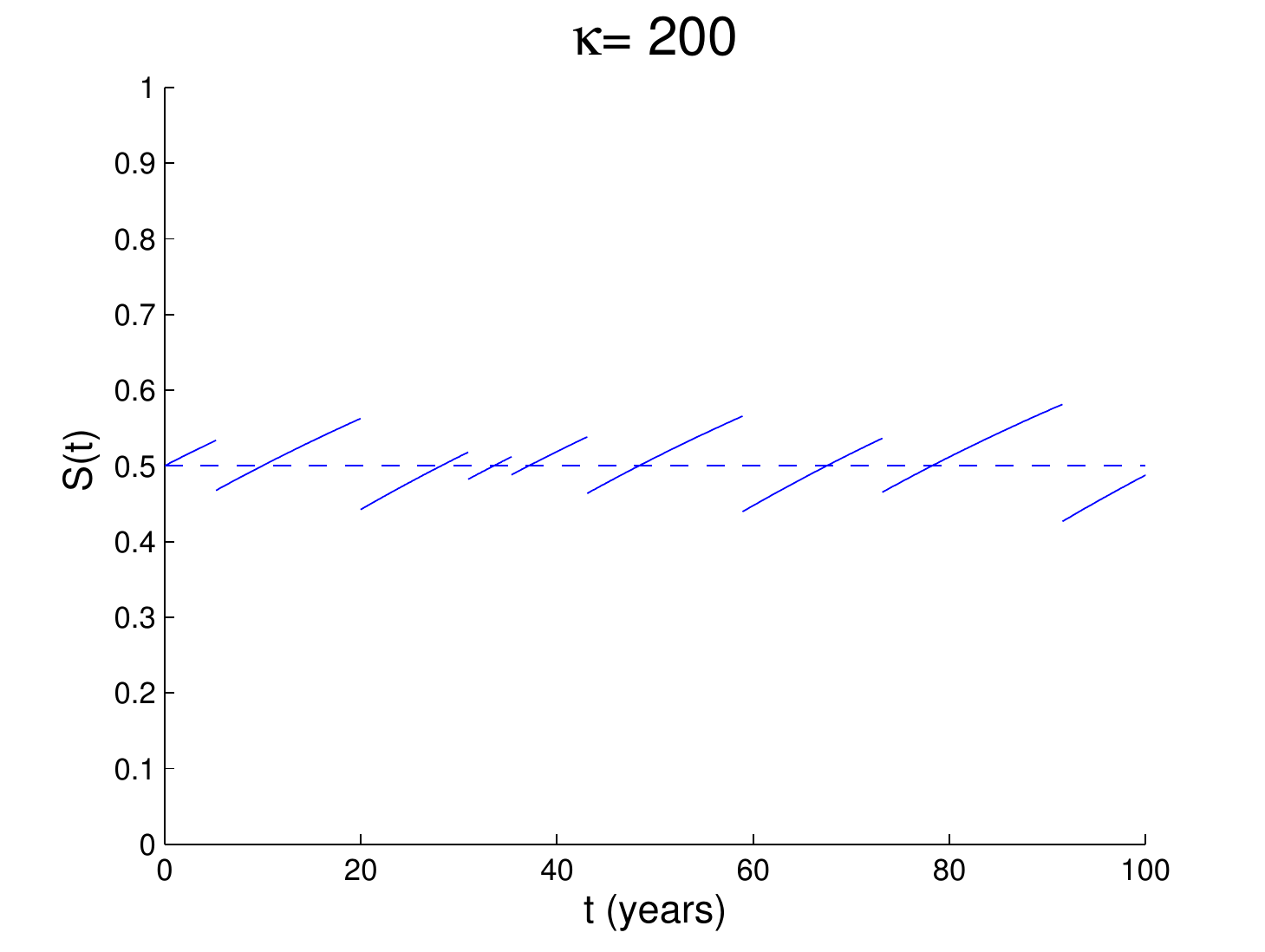}}
\end{center}
\caption{Simulation of the limiting process $S$ for the same parameter values as in the epidemics in Figure~\ref{epidsimul}.}
\label{limitingsimul}
\end{figure}

We now illustrate properties of the stationary distribution of the fraction susceptible $S^*$, both for the epidemic with $n=1,000$ and $n=10,000$, as well as for the limiting process. For the three processes, and for three different values of $\kappa$, we simulate the epidemic and limiting processes for 10,000 years and in Figure \ref{stationary-simul} we plot bar charts of the relative time spent with specified fraction susceptible.
The processes are simulated over a very long time span so that the empirical distribution of the fraction susceptible is close to the corresponding stationary distribution.  (Recalling the functional $H_{t^*}^a$ defined at the end of Section~\ref{subsecmain}, note that by standard regenerative process theory, for any fixed $a \in (0,1)$, $\frac{1}{t^*}H_{t^*}^a S \convas {\rm P}(S^* \le a)$ as $t^* \to \infty$ and, by Corollary~\ref{convfunct}, $\frac{1}{t^*}H_{t^*}^a \bar S^{(n)} \convD \frac{1}{t^*}H_{t^*}^a S$ as $n \to \infty$.)
The values of $\mu, \gamma$ and $R_0$ are the same as in Figure~\ref{epidsimul}.  (Note that the value of $\gamma$, and hence also $\lambda$ $(=R_0 \gamma)$, is the same for both values of $n$.)
The chosen values of $\kappa$ are $\kappa=1, \ 3$ and $100$, corresponding to importation of infectious individuals on average one every 75, 25 and $0.75$ years, respectively. In the plots we have also computed $f_{S^*}(s)$, the stationary distribution of the limiting process, numerically as described in Section~\ref{S-properties}.
\begin{figure}
\begin{center}
\resizebox{\figwidth}{!}{\includegraphics{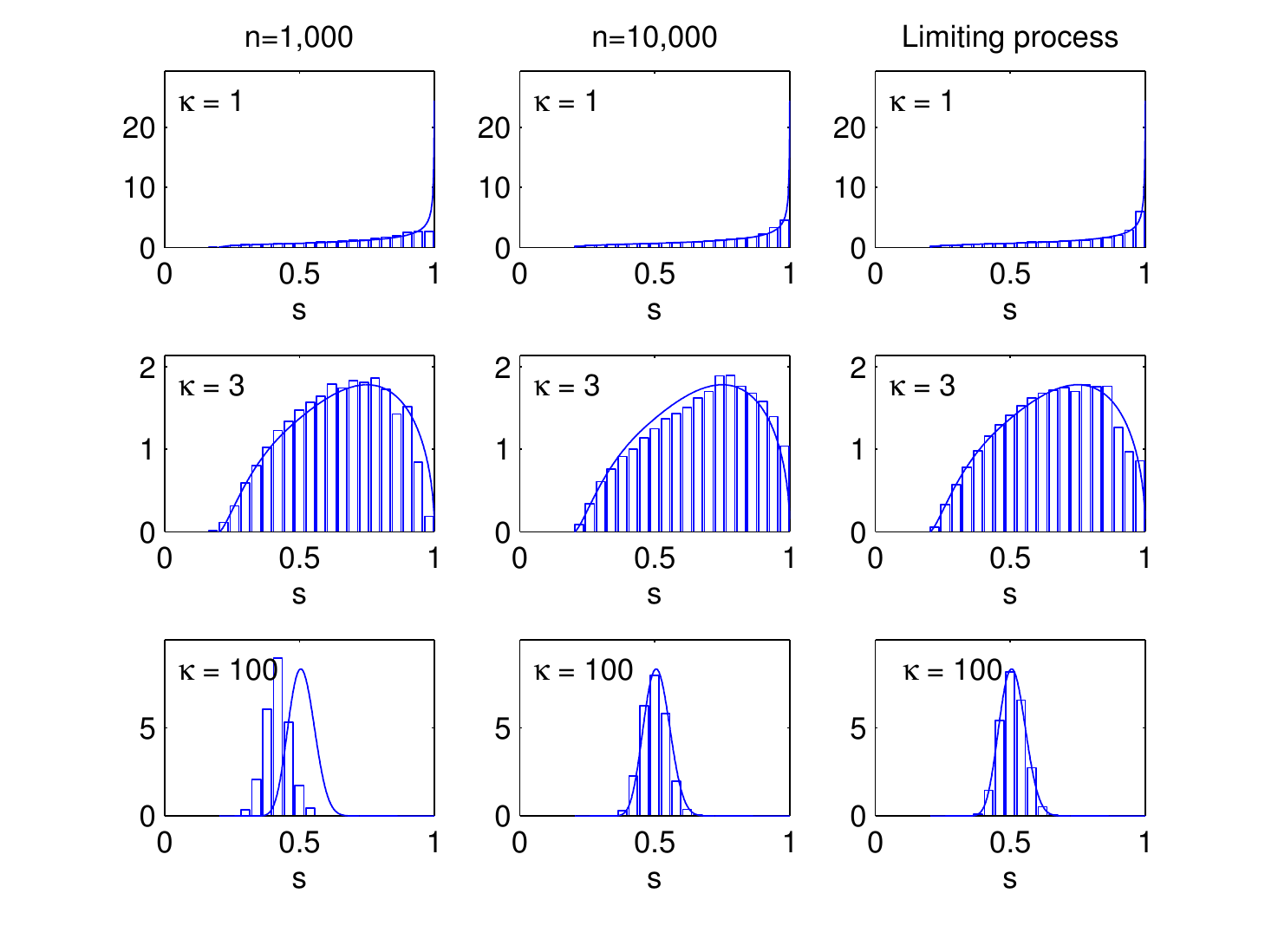}}
\end{center}
\caption{Bar charts of the relative time spent with fraction $s$ susceptible for the epidemic (with $n=1,000$ and $n=10,000$) as well as the limiting process. Also plotted is the stationary distribution of limiting process $f_{S^*}(s)$. Parameter values are: average life length equals $1/\mu=75$ years, $R_0=2$, mean infectious period $1/\gamma\approx$ 1 week and  $\kappa=1, 3$ and $100$. Bar charts are based on simulation over 10,000 years.}
\label{stationary-simul}
\end{figure}

It is seen that the bar charts from the epidemics resemble the limiting stationary distribution $f_{S^*}(s)$, except when $n=1,000$ and $\kappa=100$. When $\kappa$ is small, few outbreaks take place, so even if the outbreaks are large, the population fraction of susceptibles is close to 1 most of the time, which explains why the stationary distribution $S^*$ is concentrated at values close to 1.
For moderate values of $\kappa$, the stationary distribution has positive mass for nearly all $s$ values between $1-\tau(1)=0.2032$ (the fraction susceptible after a major outbreak starting with the entire population being susceptible) and 1. The stationary distribution is seen to be concentrated around $1/R_0$ when $\kappa$ is large, owing to the fact that a new major outbreak occurs quite soon after the population fraction of susceptibles exceeds $1/R_0$, with the effect that the size of major outbreaks is generally small. These observations imply that the stationary distribution is \emph{not} stochastically decreasing (nor increasing) in $\kappa$.

\section{Proofs}\label{proofs}
\subsection{Proof of Theorem~\ref{wconv}}\label{proof}

Let $(\Omega, \mathcal{F}, P)$ be a probability space on which is defined a homogeneous Poisson process $\eta$ on $(0,\infty)$ having rate $\mu \kappa$ and let $0<r_1<r_2<\cdots$ denote the times of the points in $\eta$.  For $n=1,2,\cdots$, let $\eta^{(n)}$ denote the point process with points at $0<r_1^{(n)}<r_2^{(n)}<\cdots$, where $r_k^{(n)}=\frac{\kappa}{n\kappa_n}r_k$ $(k=1,2,\cdots)$.  Let $E^{(n)}$ denote the epidemic process indexed by $n$.  Then $\eta^{(n)}$ gives the points in time when infectives immigrate into the population in $E^{(n)}$.  We construct $E^{(n)}$ $(n=1,2,\cdots)$ and $S$ by first conditioning on $\eta$.

The process $S$ is constructed as follows.  Recall the definition of $\tau(s)$ at~\ref{tau}.  Between the points of $\eta$, $S(t)$ increases deterministically according to the differential equation~\eqref{Sdet}.  For $k=1,2,\cdots$, $S$ has a down jump to $S(r_k-)[1-\tau(S(r_k-))]$ at time $r_k$ with probability $\max(1-(R_0 S(r_k-))^{-1},0)$ (independently for successive $k$), otherwise $S$ continues to grow according to~\eqref{Sdet}.  Thus, $S$ can be described as follows.  Let $t_1<t_2<\cdots$ be the times of the down jumps of $S$, so these form a subset of the points of $\eta$.  Let
\[
f(x,t)=1-(1-x){\rm e}^{-\mu t} \quad (0<x<1, t>0),
\]
so, for fixed $x$, the solution of~\eqref{Sdet} with $S(0)=x$ is $f(x,t)$.  Let $t_0=0$ and suppose that $s_0=S(0)$ is given.  Then, for $k=0,1,\cdots$,
\begin{equation}
\label{Sfromf}
S(t)=f(s_k,t-t_k) \quad (t_k \le t < t_{k+1}),
\end{equation}
where, for $k=1,2,\cdots$, the initial value $s_k=\tilde{s}_k(1-\tau(\tilde{s}_k))$, with
$\tilde{s}_k=S(t_k-)=f(s_{k-1},t_k-t_{k-1})$.
The precise definition of the construction of $E^{(n)}$ $(n=1,2,\cdots)$ is not relevant at this stage.

We prove Theorem~~\ref{wconv} by first proving the corresponding result for processes conditioned on $\eta$.

\begin{lem}
\label{condwconv}
Suppose that $\lim_{n \to \infty} \bar S^{(n)}(0)=s_0$.  Then, for ${\rm P}$-almost all $\eta$,
\begin{equation}
\label{weakS+S-}
\bar S^{(n)}_-|\eta \Rightarrow  S \quad \mbox{and} \quad  \bar S^{(n)}_+|\eta \Rightarrow  S \quad\mbox{as } n \to \infty.
\end{equation}
\end{lem}

In order to prove Lemma~\ref{condwconv}, we need some more notation and an extra lemma (Lemma~\ref{convd}, below).
Recall that, for $n=1,2,\cdots$, we assume $I^{(n)}(0)=0$, that $t^{(n)}_0=u^{(n)}_0=0$ and that, for $k=1,2,\cdots$, $t^{(n)}_k=\inf\left\{t\ge u^{(n)}_{k-1}:I^{(n)}(t) \ge \log n\right\}$ and $u^{(n)}_k=\inf\left\{t\ge t^{(n)}_{k}:I^{(n)}(t)=0\right\}$.  For $n=1,2,\cdots$, let $s^{(n)}_0=\bar S^{(n)}\left(u^{(n)}_0\right)$ and, for $k=1,2,\cdots$, let $s^{(n)}_k=\bar S^{(n)}\left(u^{(n)}_k\right)$.
\[
c^{(n)}_k=\min_{t^{(n)}_k \le t \le u^{(n)}_k} \bar S^{(n)}(t)\qquad \mbox{and} \qquad
\tilde{c}^{(n)}_k=\max_{t^{(n)}_k \le t \le u^{(n)}_k} \bar S^{(n)}(t)
\]

\begin{lem}
\label{convd}
Suppose that $\lim_{n \to \infty} \bar S^{(n)}(0)=s_0$.  Then the following hold for ${\rm P}$-almost all $\eta$.
\begin{enumerate}
\item[(i)]
For $k=1,2, \cdots$,\quad $u^{(n)}_k|\eta \convD t_k,\quad t^{(n)}_k|\eta \convD t_k$,\quad
$s^{(n)}_k |\eta \convD s_k$,\newline
$c^{(n)}_k |\eta \convD s_k$ \quad and \quad $\tilde{c}^{(n)}_k |\eta \convD \tilde{s}_k$ as $n \to \infty$.
\newline
\item[(ii)] For $k=0,1, \cdots$,
\begin{equation}
\label{susboundcd}
\sup_{u^{(n)}_k \le t < t^{(n)}_{k+1}} \left|\bar S^{(n)}(t)-f\left(c^{(n)}_k,t-u^{(n)}_k\right)\right||\eta \convD 0 \quad \text{as} \quad n \to \infty.
\end{equation}
\item[(iii)] $t_k \to \infty$ as $k \to \infty$.
\end{enumerate}
\end{lem}
\begin{proof} See Section~\ref{proofconvd}.
\end{proof}

\begin{proof}[Proof of Lemma~\ref{condwconv}] First note that since results concerning convergence in distribution in the Euclidean space $\mathbb{R}^k$ carry over in all essential respects to convergence in distribution in $\mathbb{R}^\infty$ (see~\cite{Bill}, page 19), the Skorohod representation theorem implies that we may assume that the convergence in Lemma~\ref{convd} holds almost surely.  Let $A \in \mathcal{F}$ be the set $\omega \in \Omega$ such that (i) for $k=1,2 \cdots,$
\begin{align*}
\lim_{n \to \infty}  u^{(n)}_k&(\omega)=t_k(\omega),\quad \lim_{n \to \infty} t^{(n)}_k(\omega)=t_k(\omega), \quad \lim_{n \to \infty} s^{(n)}_k(\omega)=s_k(\omega),\\
&\lim_{n \to \infty} c^{(n)}_k(\omega)=s_k(\omega)
\quad\mbox{and} \quad
\lim_{n \to \infty} \tilde{c}^{(n)}_k(\omega)=\tilde{s}_k(\omega);
\end{align*}
(ii) for $k=0,1,\cdots$,
\begin{equation}
\label{boundsf1}
\lim_{n\to \infty} \sup_{u^{(n)}_k (\omega) \le t < t^{(n)}_{k+1}(\omega)} \left|\bar S^{(n)}(t,\omega)-f\left(s^{(n)}_k(\omega),t-u^{(n)}_k\right)(\omega)\right|=0;
\end{equation}
and (iii) $t_k(\omega) \to \infty$ as $k \to \infty$.
Then ${\rm P}(A|\eta)=1$ for ${\rm P}$-almost all $\eta$.

For $g,h \in D[0,\infty)$, $d(g,h)$ denotes the distance between $g$ and $h$ in the Skorohod metric (see~\cite{EK86}, Chapter 3.5).  Let $\eta$ satisfy ${\rm P}(A|\eta)=1$.  We show that for all $\omega \in A$,
\[
\lim_{n \to \infty} d(\bar S^{(n)}_-(\omega), S(\omega))=0 \quad \mbox{and} \quad \lim_{n \to \infty} d(\bar S^{(n)}_+(\omega), S(\omega))=0.
\]
It then follows that, under the Skorohod metric, both $\bar S^{(n)}_-|\eta$ and $\bar S^{(n)}_+|\eta$ converge almost surely to $S$, which implies~\eqref{weakS+S-}.

By Proposition 5.3 on page 119 of~\cite{EK86}, to show that $d(g_n,g) \to 0$ as $n \to \infty$ it is sufficient to show that for each $T>0$, there exists a sequence $(\lambda_n)$  of strictly increasing functions mapping $[0,\infty)$ onto $[0,\infty)$ so that
\begin{equation}
\label{skconvcond1}
\lim_{n \to \infty} \sup_{0\le t \le T}|\lambda_n(t)-t|=0
\end{equation}
and
\begin{equation}
\label{skconvcond2}
\lim_{n \to \infty} \sup_{0\le t \le T}|g(\lambda_n(t))-g_n(t)|=0.
\end{equation}

For ease of exposition we now suppress dependence on $\omega$.
Fix $T\ge t_1$ and let $m=\max\{k:t_k \le T\}$, so $1 \le m < \infty$..  Then, there exists $\delta>0$ such that $u^{(n)}_m<T+\delta$ for all sufficiently large $n$.  For such $n$, let $\lambda^{(n)}_-$ be the piecewise-linear function joining the points $(0,0), (t^{(n)}_1,t_1),\cdots,(t^{(n)}_m,t_m),(T+\delta,T+\delta)$, with  $\lambda^{(n)}_-(t)=t$ for $t>T+\delta$.  Similarly,  let $\lambda^{(n)}_+$ be the piecewise-linear function joining the points $(0,0), (u^{(n)}_1,t_1),\cdots,(u^{(n)}_m,t_m),(T+\delta,T+\delta)$, with  $\lambda^{(n)}_+(t)=t$ for $t>T+\delta$.  (Note that $\lambda^{(n)}_+(t)\le\lambda^{(n)}_-(t)$ with strict inequality for $t \in (0,T+\delta)$.)
The functions $\lambda^{(n)}_-$ and $\lambda^{(n)}_+$ are strictly increasing and satisfy~\eqref{skconvcond1}, since $t^{(n)}_k \to t_k$ and $u^{(n)}_k \to t_k$ as $n \to \infty$ ($k=1,2,\cdots,m)$.  Thus, to complete the proof we show that
\begin{equation}
\label{skcond2-}
\lim_{n \to \infty} \sup_{0\le t \le T}\left|S\left(\lambda^{(n)}_-(t)\right)-\bar S^{(n)}_-(t)\right|=0
\end{equation}
and
\begin{equation}
\label{skcond2+}
\lim_{n \to \infty} \sup_{0\le t \le T}\left|S\left(\lambda^{(n)}_+(t)\right)-\bar S^{(n)}_+(t)\right|=0.
\end{equation}

Considering~\eqref{skcond2-} first, note that for $k=1,2,\cdots,m$, since $S\left(\lambda^{(n)}_-(t)\right)$ is increasing on $[t^{(n)}_k , u^{(n)}_k]$ and $\bar S^{(n)}_-(t)=c^{(n)}_k$ for all $t \in [t^{(n)}_k , u^{(n)}_k)$,

\begin{align}
\sup_{t^{(n)}_k\le t < u^{(n)}_k}&
\left|S\left(\lambda^{(n)}_-(t)\right)-\bar S^{(n)}_-(t)\right|\nonumber\\&\le
\max\left\{\left|S\left(\lambda^{(n)}_-(t^{(n)}_k)\right)-c^{(n)}_k\right|,\left|S\left(\lambda^{(n)}_-(u^{(n)}_k)\right)-c^{(n)}_k\right|\right\}\nonumber\\
&\to 0 \quad \mbox{as } n \to \infty, \label{skcond2-a}
\end{align}
since $\lambda^{(n)}_-(t^{(n)}_k)=t_k$, $c^{(n)}_k \to s_k=S(t_k)$ and
$\lambda^{(n)}_-(u^{(n)}_k) \to t_k$ (as $u^{(n)}_k \to t_k$ and $\lambda^{(n)}_-$ is continuous), so
$S\left(\lambda^{(n)}_-(u^{(n)}_k)\right) \to s_k$ as $S$ is right-continuous.

Also, for $k=1,2,\cdots,m-1$,
\[
\sup_{u^{(n)}_k \le t < t^{(n)}_{k+1}}\left|S\left(\lambda^{(n)}_-(t)\right)-\bar S^{(n)}_-(t)\right|\le A(n,k)+B(n,k),
\]
where
\[
A(n,k)=\sup_{u^{(n)}_k \le t < t^{(n)}_{k+1}}\left|S\left(\lambda^{(n)}_-(t)\right)-f\left(s^{(n)}_k,t-u^{(n)}_k\right)\right|
\]
and
\[
B(n,k)=\sup_{u^{(n)}_k \le t < t^{(n)}_{k+1}}\left|f\left(s^{(n)}_k,t-u^{(n)}_k\right)-\bar S^{(n)}_-(t)\right|.
\]

Now $\lambda^{(n)}_-(t) \in [t_k, t_{k+1})$ for $t \in [u^{(n)}_k ,t^{(n)}_{k+1})$, so using~\eqref{Sfromf},
\begin{align}
A(n,k)&=\sup_{u^{(n)}_k \le t < t^{(n)}_{k+1}}\left|f\left(s_k,\lambda^{(n)}_-(t)-t_k\right)-f\left(s^{(n)}_k,t-u^{(n)}_k\right)\right|\nonumber\\
&\le\sup_{u^{(n)}_k \le t < t^{(n)}_{k+1}}\left|f\left(s_k,\lambda^{(n)}_-(t)-t_k\right)-f\left(s_k,t-u^{(n)}_k\right)\right|\nonumber\\
&\qquad\qquad+\sup_{u^{(n)}_k \le t < t^{(n)}_{k+1}}\left|f\left(s_k,t-u^{(n)}_k\right)-f\left(s^{(n)}_k,t-u^{(n)}_k\right)\right|.\label{Ankbound}
\end{align}
A simple argument using the mean value theorem shows that, for $x\in [0,1]$ and $t,t'\ge0$,
\begin{equation}
\label{fbound1}
|f(x,t)-f(x,t')|\le(1-x)\mu|t-t'|.
\end{equation}
Now
\begin{eqnarray*}
\sup_{u^{(n)}_k \le t < t^{(n)}_{k+1}}\left|\lambda^{(n)}_-(t)-t_k-(t-u^{(n)}_k)\right| &\le& \sup_{u^{(n)}_k \le t < t^{(n)}_{k+1}}\left| \lambda^{(n)}_-(t)-t\right|+\left|t_k-u^{(n)}_k\right|\\
&\to&0\quad \mbox{as } n \to \infty,
\end{eqnarray*}
as $\lambda^{(n)}_-$ satisfies~\eqref{skconvcond1} and  $u^{(n)}_k \to t_k$ as $n \to \infty$. It then follows using~\eqref{fbound1} that the first term on the right hand side of~\eqref{Ankbound} tends to $0$ as  $n \to \infty$.
Also, for $x,y \in [0,1]$ and $t \ge 0$,
\[
f(x,t)-f(y,t)=(y-x){\rm e}^{-\mu t},
\]
so the second term on the right hand side of~\eqref{Ankbound} tends to $0$ as $n \to \infty$, since $s^{(n)}_k \to s_k$ as $n \to \infty$.  Thus, $A(n,k) \to 0$ as $n \to \infty$.

Note that $\bar S^{(n)}_-(t)-=\bar S^{(n)}(t)$ for $t \in [u^{(n)}_k,t^{(n)}_{k+1})$, so~\eqref{boundsf1} implies that $B(n,k)$ also converges to $0$ as $n \to \infty$, whence
\begin{equation}
\label{skcond2-b}
\sup_{u^{(n)}_k \le t < t^{(n)}_{k+1}}\left|S\left(\lambda^{(n)}_-(t)\right)-\bar S^{(n)}_-(t)\right| \to 0 \quad \mbox{as } n\to \infty.
\end{equation}
Combining~\eqref{skcond2-a} and \eqref{skcond2-b} yields that,
\begin{equation}
\label{skcond2-c}
\lim_{n \to \infty} \sup_{u_1^{(n)} \le t < t^{(n)}_m} \left|S\left(\lambda^{(n)}_-(t)\right)-\bar S^{(n)}_-(t)\right|=0.
\end{equation}

A similar argument to the derivation of~\eqref{skcond2-b} yields
\[
\lim_{n \to \infty} \sup_{0 \le t < u^{(n)}_1} \left|S\left(\lambda^{(n)}_-(t)\right)-\bar S^{(n)}_-(t)\right|=
\lim_{n \to \infty} \sup_{t^{(n)}_m \le t \le T} \left|S\left(\lambda^{(n)}_-(t)\right)-\bar S^{(n)}_-(t)\right|=
0,
\]
which together with~\eqref{skcond2-c} yields~\eqref{skcond2-}, as required.

The proof of~\eqref{skcond2+} is similar to that of~\eqref{skcond2-} and hence omitted.
\end{proof}

\begin{proof}[Proof of Theorem~\ref{wconv}]
We prove the result for $\bar S^{(n)}_-$.  The proof for $\bar S^{(n)}_+$ is identical.
Recall that if $X_n$ $(n=1,2,\cdots)$ and $X$ are random elements of $D[0,\infty)$ then $X_n \Rightarrow X$ as $n \to \infty$ if and only if ${\rm E}\left[f(X_n)\right] \to \left[f(X)\right]$ as $n \to \infty$ for all bounded, uniformly continuous functions $f:D[0,\infty) \to \mathbb{R}$ (see, for example,~\cite{EK86}, Chapter 3, Theorem 3.1).
Let $f:D[0,\infty) \to \mathbb{R}$ be any such function.  Then Lemma~\ref{condwconv} implies that, for
${\rm P}$-almost all $\eta$,
\[
\lim_{n \to \infty}{\rm E}\left[f(\bar S^{(n)}_-)|\eta\right] = {\rm E}\left[f(S)|\eta\right].
\]
Hence, by the dominated convergence theorem,
\begin{eqnarray*}
\lim_{n \to \infty}{\rm E}\left[f(\bar S^{(n)}_-)\right] &=&
\lim_{n \to \infty}{\rm E}_\eta\left[{\rm E}\left[f(\bar S^{(n)}_-)|\eta\right]\right]\\
&=&{\rm E}_\eta\left[\lim_{n \to \infty}{\rm E}\left[f(\bar S^{(n)}_-)|\eta\right]\right]\\
&=&{\rm E}_\eta\left[{\rm E}\left[f(S)|\eta\right]\right]\\
&=&{\rm E}\left[f(S)\right].
\end{eqnarray*}
This holds for all bounded, uniformly continuous $f:D[0,\infty) \to \mathbb{R}$, so $\bar S^{(n)}_- \Rightarrow  S$ as $n \to \infty$, as required.
\end{proof}

We end this subsection by showing that the occupancy time functional $H_{t^*}^a$, defined at~\eqref{occfunct}, satisfies
${\rm P}\left(S \in C_{H_{t^*}^a}\right)=1$.  Recall that $t_1<t_2<\cdots$ denote the jump times of $S$.  Let $v_1=\inf\{t \ge 0: S(t)=a\}$ and, for $k=2,3,\cdots$, let $v_k=\inf\{t > v_{k-1}: S(t)=a\}$.  Let $C \in \mathcal{F}$ be the set of $\omega \in \Omega$ such that $t_k(\omega)$ (and hence also $v_k(\omega)$) tends to $\infty$ as $k \to \infty$.  Then, by Lemma~\ref{convd} (iii), ${\rm P}(C)=1$.  We show that if $g_n \in D[0,\infty)$ $(n=1,2,\cdots)$ and $\lim_{n \to \infty} d(g_n,S(\omega))=0$, then
$\lim_{n \to \infty} H_{t^*}^ag_n=H_{t^*}^aS(\omega)$, for $\omega \in C$, whence ${\rm P}\left(S \in C_{H_{t^*}^a}\right)=1$.

Suppose that $\omega \in C$.  Dropping the explict dependence of $S$ on $\omega$, since $\lim_{n \to \infty} d(g_n,S)=0$, by Proposition 5.3 on page 119 
of~\cite{EK86}, there exists a sequence $(\lambda_n)$ of strictly increasing functions mapping $[0, \infty)$ onto $[0, \infty)$ such that 
\begin{equation}
\label{skorconv}
\lim_{n \to \infty} \sup_{0\le t \le t^*}|\lambda_n(t)-t|=0
\quad \mbox{and} \quad
\lim_{n \to \infty} \sup_{0\le t \le t^*}|S(\lambda_n(t))-g_n(t)|=0.
\end{equation}
Now
\begin{equation*}
\left|H_{t^*}^a g_n - H_{t^*}^a S\right|=\left| \int_0^{t^*} 1_{\{g_n(t)\le a\}}-1_{\{S(t)\le a\}}\,{\rm d}t\right|\\
\le A_n +B_n,
\end{equation*}
where
\[
A_n=\int_0^{t^*}\left|1_{\{g_n(t)\le a\}}-1_{\{S\left(\lambda_n(t)\right)\le a\}}\right|\,{\rm d}t
\]
and
\[
B_n=\int_0^{t^*}\left|1_{\{S\left(\lambda_n(t)\right)\le a\}}-1_{\{S(t)\le a\}}\right|\,{\rm d}t.
\]
Let $D=[0,t^*]\cap\left(\left\{t_1,t_2,\cdots\right\}\cup\left\{v_1,v_2,\cdots\right\}\right)$. Then
$D$ has Lebesgue measure zero and $1_{\{S\left(\lambda_n(t)\right)\le a\}}-1_{\{S(t)\le a\}}
\to 0$ as $n \to \infty$, for $t \in [0,t^*] \setminus D$, since $\lim_{n \to \infty}\lambda_n(t)=t$, by the first equation in~\ref{skorconv}, and $S$ is continuous at such $t$.  Thus $\lim_{n \to \infty} B_n =0$ by the dominated convergence theorem.  A similar argument, using in addition the second equation in~\ref{skorconv}, shows that $\lim_{n \to \infty} A_n =0$.  Thus, $\lim_{n \to \infty}H_{t^*}^a g_n = H_{t^*}^a S$, as required.
\subsection{Proof of Lemma~\ref{convd}}
\label{proofconvd}

We prove Lemma~\ref{convd} by splitting the SIR-D-I epidemic process $E^{(n)}$ into cycles, where now
a cycle begins at the end of a major outbreak and finishes at the end of the following major outbreak.  Thus a cycle consists of two stages: stage 1, during which the susceptible population grows approximately deterministically until there are at least $\log n$ infectives present; and stage 2, comprising the major outbreak caused by these $\log n$ infectives, during which the susceptible population crashes.

Recall that, as $n \to \infty$, the point process $\eta^{(n)}$, describing immigration times of infectives in $E^{(n)}$ converges almost surely to the point process $\eta$ governing times when down jumps \emph{may} occur in the limiting process $S$. Lemma~\ref{minoreps} considers the initial stage 1 and shows, using birth-and-death processes that sandwich the process of infectives, that for ${\rm P}$-almost all $\eta$,
as $ n \to \infty$, for successive importations of infectives until a major outbreak occurs, the probability a given importation triggers a major outbreak converges to the probability that the corresponding importation results in a down jump in the limiting process $S$.  Consequently, the time until
there are at least $\log n$ infectives in $E^{(n)}$ converges weakly to the time of the first down jump in $S$, since $\eta^{(n)}$ converges almost surely to $\eta$.  Further,
application of the law of large numbers for density dependent population processes (~\cite{EK86}, Chapter 11) shows that up until the first down jump of $S$, the scaled process of susceptibles, $\bar S^{(n)}= n^{-1} S$, converges weakly in the uniform metric to $S$, since minor epidemics infect order $o_p(n)$ individuals.

Lemmas~\ref{majoreps0} and~\ref{majoreps} concern the limiting size and duration of a typical major outbreak.  Lemma~\ref{majoreps0} considers outbreaks in which the initial number of infectives is of exact order $n$, for which the above-mentioned law of large numbers is applicable.  This is then used to prove
Lemma~\ref{majoreps}, which considers major outbreaks triggered by $\log n$ infectives.  Finally, Lemma~\ref{convd} follows easily by induction using Lemmas~\ref{minoreps} and~\ref{majoreps}, since
$E^{(n)}$ is Markov.

The proof involves extensive use of birth-and-death processes that bound the process of infectives in
the epidemic model (cf.~\cite{Whittle55}).  We first give some notation concerning birth-and-death processes
and then state a lemma, proved in Appendix~\ref{appendix1}, concerning properties of sequences of such processes.

Let $Z_{\alpha,\beta,k}=\{Z_{\alpha,\beta,k}(t):t\ge 0\}$ denote a linear birth-and-death process, with $Z_{\alpha,\beta,k}(0)=k$, birth rate $\alpha$ and death rate $\beta$.  For $x>k$, let $\tau_{\alpha,\beta,k}(x)=\inf\{t>0:Z_{\alpha,\beta,k}(t)\ge x\}$, where $\tau_{\alpha,\beta,k}(x)=\infty$ if $Z_{\alpha,\beta,k}(t)< x$ for all $t>0$. (Throughout the paper we adopt the convention that the hitting time of an event is infinite if the event never occurs.) Let $\tau_{\alpha,\beta,k}(0)=
\inf\{t>0: Z_{\alpha,\beta,k}(t)=0\}$ denote the duration of $Z_{\alpha,\beta,k}$.  For $t\ge0$, let $B_{\alpha,\beta,k}(t)$ denote the total number of births during $(0,t]$ in $Z_{\alpha,\beta,k}$, and let $B_{\alpha,\beta,k}(\infty)$ denote the total progeny of $Z_{\alpha,\beta,k}$, not including the $k$ ancestors.  Further, for $x>0$, let $\hat{\tau}_{\alpha,\beta,k}(x)=\inf\{t>0:B_{\alpha,\beta,k}(t)\ge x\}$.

\begin{lem}
\label{bdresults}
Suppose that $\alpha_n=a\beta_n$ $(n=1,2,\cdots)$, where $a>0$ is constant and $\log n/\beta_n \to 0$ as $n \to \infty$.
\begin{enumerate}
\item[(a)]  If $a<1$, then
\begin{enumerate}
\item[(i)]  for all $t>0$,
\[
\lim_{n \to \infty}{\rm P}\left(\tau_{\alpha_n,\beta_n,1}(0)>t\right)=0;
\]
\item[(ii)] $\lim_{n\to\infty}{\rm P}\left(\tau_{\alpha_n,\beta_n,1}(\log n)= \infty\right)=1$; and
\item[(iii)] for any $c>0$,
\[
\tau_{\alpha_n,\beta_n,\left \lceil{c n}\right \rceil}(0) \convp 0 \quad\mbox{as }  n \to \infty.
\]
\end{enumerate}
\item[(b)]  If $a>1$, then
\begin{enumerate}
\item[(i)] $\lim_{n\to\infty}{\rm P}\left(\tau_{\alpha_n,\beta_n,1}(\log n)< \tau_{\alpha_n,\beta_n,1}(0)\right)=1-\frac{1}{a},\newline \lim_{n\to\infty}{\rm P}\left(\tau_{\alpha_n,\beta_n,1}(0)< \tau_{\alpha_n,\beta_n,1}(\log n)\right)=\frac{1}{a}$;
\item[(ii)] $\min\left(\tau_{\alpha_n,\beta_n,1}(\log n), \tau_{\alpha_n,\beta_n,1}(0)\right)\convp 0$ as $n \to \infty$;
\item[(iii)] $\lim_{n\to\infty}{\rm P}\left(B_{\alpha_n,\beta_n,1}\left(\min\left\{\tau_{\alpha_n,\beta_n,1}(\log n), \tau_{\alpha_n,\beta_n,1}(0)\right\}\right)<n^{\frac{1}{3}}\right)=1$; and
\item[(iv)] for any $c>0$,
\[
\hat{\tau}_{\alpha_n,\beta_n,\left \lceil{\log n}\right \rceil}(cn) \convp 0 \quad\mbox{as }  n \to \infty.
\]
\end{enumerate}
\end{enumerate}
\end{lem}

Before proceeding some more notation is required.  For  $k=1,2,\cdots$, let $\chi_k=1_{\{S(r_k)<S(r_k-)\}}$ be the indicator function of the event that
the $k$th point in $\eta$ yields a down jump in $S$.  For $n=1,2,\cdots$ and $k=1,2,\cdots$, let
$w_k^{(n)}=\inf\left\{t \ge r_k^{(n)}: I^{(n)}(t)\ge \log n \mbox{ or } I^{(n)}(t)=0\right\}$ and
$\chi_k^{(n)}=1_{\{I(w_k^{(n)}) \ge \log n\}}$.

\begin{lem}
\label{minoreps}
Suppose that $\bar S^{(n)}(0) \convp s_0$ as $n \to \infty$.  Then the following hold for ${\rm P}$-almost all $\eta$.
\begin{enumerate}
\item[(i)] For $k=1,2,\cdots$,
\begin{align*}
\lim_{n \to \infty}{\rm P}\left(\chi_k^{(n)}=1 \right.&\left.\mbox{ and } \chi_i^{(n)}=0 \mbox{ for all } i < k | \eta\right)\\
&={\rm P}(\chi_k=1 \mbox{ and } \chi_i=0 \mbox{ for all } i < k | \eta).
\end{align*}
\item[(ii)] For $k=1,2,\cdots$, as $n \to \infty$,
\[
\sup_{0 \le t < w_k^{(n)}}\left|\bar S^{(n)}(t)-f(s_0,t)\right|1_{\{\chi_k^{(n)}=1 \mbox{ and } \chi_i^{(n)}=0 \mbox{ for all } i < k \}}
|\eta \convD 0.
\]
\item[(iii)] For $k=1,2,\cdots$, as $n \to \infty$,
\[
w_k^{(n)}1_{\{\chi_k^{(n)}=1 \mbox{ and } \chi_i^{(n)}=0 \mbox{ for all } i < k \}}\convD
r_k 1_{\{\chi_k=1 \mbox{ and } \chi_i=0 \mbox{ for all } i < k \}}.
\]
\end{enumerate}
\end{lem}

\begin{proof}
For ease of presentation we suppress explicit conditioning on $\eta$ in the proof.  First note that ${\rm P}\left(S(r_1-)=R_0^{-1}\right)=0$, since $r_1$ is a realisation of a continuous random variable.  Assume without loss of generality that there is no recovered individual at time $t=0$.  For $t \ge 0$, let $S^{(n)}_0(t)$ be the number of susceptibles at time $t$ under the assumption that the immigration rate for susceptibles is $\mu n (1-\kappa_n)$ and the immigration rate for infectives is $0$, and let $\bar S^{(n)}_0(t)=S^{(n)}_0(t)/n$.   Then, for any $t > 0$, application of Theorem 11.2.1 of~\cite{EK86} (using the more general definition of a density dependent family given by equation (11.1.13) of that book) yields that, for any $\epsilon>0$,
\begin{equation}
\label{kurtzsus0}
\lim_{n \to \infty} {\rm P}\left(\sup_{0 \le u \le t}\left|\bar S^{(n)}_0(u)-f(s_0,u)\right| < \epsilon\right)=1.
\end{equation}

Recall that $E^{(n)}$ denote the epidemic process with average population size $n$.
Consider the epidemic initiated by the immigration of an infective at time $r_1^{(n)}$ in $E^{(n)}$ and let $\hat{s}_1^{(n)}=\bar S^{(n)}(r_1^{(n)})$.  For ease of exposition, translate the
time axis of $E^{(n)}$ so that the origin corresponds to $r_1^{(n)}$.
With this new time origin, $\left\{I^{(n)}(t):t \ge 0\right\}$ can be approximated by a linear birth-and-death process $\left\{\tilde{I}^{(n)}(t):t \ge 0\right\}$ having death rate $\gamma_n+\mu$ and (random) time-dependent birth rate  given by $\lambda_n \bar S^{(n)}_0(t)$.  This approximation ignores depletion in the number of susceptibles owing to infection, 
so $\left\{\tilde{I}^{(n)}(t):t \ge 0\right\}$ is an upper bound for $\left\{I^{(n)}(t):t \ge 0\right\}$.

Let $\hat{s}_1=f(s_0,r_1)$ and fix $\epsilon \in (0, \hat{s}_1)$.  Note that, with the change of origin, $\bar S^{(n)}_0(0)=\hat{s}_1^{(n)}$.  Then, using~\eqref{kurtzsus0}, for any $\delta \in (0,1)$, there exists $\hat{t}=\hat{t}(\epsilon,\delta)>0$ and
$n_0=n_0(\epsilon,\delta)$ such that
\begin{equation}
\label{kurtzsushat}
{\rm P}\left(\sup_{0 \le t \le \hat{t}}\left|\bar S^{(n)}_0(t)-\hat{s}_1\right| < \frac{\epsilon}{2}\right)\ge 1-\frac{\delta}{2} \quad \mbox{ for all }n \ge n_0.
\end{equation}
For $n \ge n_0$ and $0 \le t \le \hat{t}$ , with probability at least $1-\frac{\delta}{2}$, the process $\left\{\tilde{I}^{(n)}(t):t \ge 0\right\}$ is bounded below and above by the birth-and-death processes $Z_{\tilde{\alpha}_n^-(\epsilon),\beta_n,1}$ and $Z_{\tilde{\alpha}_n^+(\epsilon),\beta_n,1}$, respectively, where $\tilde{\alpha}_n^-(\epsilon)=\lambda_n(\hat{s}_1-\frac{\epsilon}{2}), \tilde{\alpha}_n^+(\epsilon)=\lambda_n(\hat{s}_1+\frac{\epsilon}{2})$ and $\beta_n=\gamma_n+\mu$.  Further, since
$\lim_{n \to \infty} \tilde{\alpha}_n^-(\epsilon)/\beta_n=R_0(\hat{s}_1-\frac{\epsilon}{2})$, for all sufficiently large $n$, the birth-and-death process $Z_{\tilde{\alpha}_n^-(\epsilon),\beta_n,1}$ is bounded below by the birth-and-death process $Z_{\alpha_n^-(\epsilon),\beta_n,1}$, where $\alpha_n^-(\epsilon)=R_0(\hat{s}_1-\epsilon)\beta_n$.  Similarly,
for all sufficiently large $n$, the birth-and-death process $Z_{\tilde{\alpha}_n^+(\epsilon),\beta_n,1}$ is bounded above by the birth-and-death process $Z_{\alpha_n^+(\epsilon),\beta_n,1}$, where $\alpha_n^+(\epsilon)=R_0(\hat{s}_1+\epsilon)\beta_n$.

Suppose first that $R_0 \hat{s}_1 <1$.  Then for all $\epsilon \in (0,\epsilon_0)$, where $\epsilon_0=R_0^{-1}-\hat{s}_1$, the birth-and-death process $Z_{\alpha_n^+(\epsilon),\beta_n,1}$ is subcritical, so by Lemma~\ref{bdresults}(a)(i), for all $t>0$,
\begin{equation}
\label{z+dur}
\lim_{n \to \infty} {\rm P}\left(\tau_{\alpha_n^+(\epsilon),\beta_n,1}(0)\le t\right)=1.
\end{equation}
Setting $t=\hat{t}$ shows that, for all sufficiently large $n$, with probability at least $1-\delta$, $\{\tilde{I}^{(n)}(t):t \ge 0\}$, and hence also $\{I^{(n)}(t):t \ge 0\}$, is bounded above by $Z_{\alpha_n^+(\epsilon),\beta_n}$ throughout its entire lifetime. Thus,
\begin{eqnarray*}
\liminf_{n \to \infty} {\rm P}\left(\chi_1^{(n)}=0\right) &\ge&
\liminf_{n \to \infty} {\rm P}\left(Z_{\alpha_n^+(\epsilon),\beta_n,1}(t) < \log n \quad \mbox{for all } t \ge 0\right)-\delta\\
&=&1-\delta,
\end{eqnarray*}
by Lemma~\ref{bdresults}(a)(ii).
Hence, since $\delta\in(0,1)$ is arbitrary,
\begin{equation*}
\lim_{n \to \infty}{\rm P}\left(\chi_1^{(n)}=0\right)=1={\rm P}\left(\chi_1=0\right).
\end{equation*}
Let $D^{(n)}=\inf\left\{t>0:I^{(n)}(t)\ge \log n\mbox{ or } I^{(n)}(t)=0\right\}$.  Then it follows using~\eqref{z+dur} that $D^{(n)} \convp 0$ as $n \to \infty$.

Suppose instead that $R_0 \hat{s}_1 >1$. Fix $\epsilon \in (0,\epsilon_1)$, where $\epsilon_1=\hat{s}_1-R_0^{-1}$, and $\delta \in (0,1)$. Then, similar to above, there exists $t_1$ such that, for all sufficiently large $n$, with probability at least $1-\frac{\delta}{2}$, $\left\{\tilde{I}^{(n)}(t):t \ge 0\right\}$ is bounded above and below by $Z_{\alpha_n^+(\epsilon),\beta_n}$
$Z_{\alpha_n^-(\epsilon),\beta_n}$, respectively, throughout the interval $[0,t_1]$.  For $x>0$, let $\tilde{\tau}^{(n)}(x)=\inf\left\{t>0:\tilde{I}^{(n)}(t)\ge x\right\}$, $\tilde{\tau}^{(n)}(0)=\inf\left\{t>0:\tilde{I}^{(n)}(t)=0\right\}$ and
$\tilde{D}^{(n)}=\inf\left\{t>0:\tilde{I}^{(n)}(t)\ge \log n\mbox{ or } \tilde{I}^{(n)}(t)=0\right\}$.  Note that the birth-and-death processes $Z_{\alpha_n^-(\epsilon),\beta_n,1}$ and $Z_{\alpha_n^+(\epsilon),\beta_n,1}$ are both supercritical. Then, by Lemma~\ref{bdresults}(b)(ii), for all sufficiently large $n$, the process $\left\{\tilde{I}^{(n)}(t):t \ge 0\right\}$ is bounded below and above by $Z_{\alpha_n^+(\epsilon),\beta_n}$
$Z_{\alpha_n^-(\epsilon),\beta_n}$, respectively, throughout the interval $[0,\tilde{D}^{(n)}]$. Using Lemma~\ref{bdresults}(b)(i), it then follows that
\begin{align}
\label{lowbound}
\liminf_{n \to \infty}{\rm P}&\left(\tilde{\tau}^{(n)}(\log n)< \tilde{\tau}^{(n)}(0)\right) \nonumber\\
&\ge\liminf_{n \to \infty}{\rm P}\left(\tau_{\alpha_n^-(\epsilon),\beta_n,1}(\log n) < \tau_{\alpha_n^-(\epsilon),\beta_n,1}(0)\right)-\delta\nonumber\\
&=1-\frac{1}{R_0(\hat{s}_1-\epsilon)}-\delta
\end{align}
and
\begin{align}
\label{upbound}
\limsup_{n \to \infty}{\rm P}&\left(\tilde{\tau}^{(n)}(\log n)< \tilde{\tau}^{(n)}(0)\right) \nonumber\\
&\le\limsup_{n \to \infty}{\rm P}\left(\tau_{\alpha_n^+(\epsilon),\beta_n,1}(\log n) < \tau_{\alpha_n^-(\epsilon),\beta_n,1}(0)\right)-\delta\nonumber\\
&=1-\frac{1}{R_0(\hat{s}_1+\epsilon)}-\delta.
\end{align}
Letting both $\epsilon$ and $\delta$ converge down to $0$ in~\eqref{lowbound} and~\eqref{upbound} yields
\begin{eqnarray}
\label{intildeconv}
\lim_{n \to \infty} {\rm P}\left(\tilde{\tau}^{(n)}(\log n)< \tilde{\tau}^{(n)}(0)\right)&=&
1-\frac{1}{R_0\hat{s}_1}\nonumber\\
&=&{\rm P}\left(\chi_1=1\right).
\end{eqnarray}
Further, using Lemma~\ref{bdresults}(b)(ii), it follows that
\begin{equation}
\label{dntildeconv}
\tilde{D}^{(n)}\convp 0 \quad \mbox{as } n \to \infty.
\end{equation}

Recall that $\left\{\tilde{I}^{(n)}(t):t \ge 0\right\}$ is an upper bound for $\left\{I^{(n)}(t):t \ge 0\right\}$.  We now show that the probability that the two processes coincide over $[0, \tilde{D}^{(n)}]$ converges to one as $n\to \infty$.  In $\left\{\tilde{I}^{(n)}(t):t \ge 0\right\}$ births occur at time-dependent rate  $\lambda_n \bar S^{(n)}_0(t)$, whilst in $\left\{I^{(n)}(t):t \ge 0\right\}$ infections occur at time-dependent rate  $\lambda_n \bar S^{(n)}(t)$.  Now $\bar S^{(n)}_0(t) \ge \bar S^{(n)}(t)$ for all $t \ge 0$, almost surely, so the two processes can be coupled by using an independent sequence $U_1,U_2,\cdots$ of independent and identically distributed random variables that are uniformly distributed on $(0,1)$, with the $i$th birth in  $\left\{\tilde{I}^{(n)}(t):t \ge 0\right\}$ (which occurs at time $t_i$ say) yielding an infection in $\left\{I^{(n)}(t):t \ge 0\right\}$ if and only if $U_i \le \bar S^{(n)}(t_i)/\bar S^{(n)}_0(t_i)$.

For $n=1,2,\cdots$ and $t>0$, let $\tilde{B}^{(n)}(t)$ be the total number of births in $\left\{\tilde{I}^{(n)}(t):t \ge 0\right\}$ during $(0,t]$. Recall that the probability that  $\left\{\tilde{I}^{(n)}(t):t \ge 0\right\}$ is sandwiched between the supercritical birth-and-death processes $Z_{\alpha_n^+(\epsilon),\beta_n}$ and $Z_{\alpha_n^-(\epsilon),\beta_n}$ throughout $[0, \tilde{D}^{(n)}]$ converges to one as $n\to \infty$.  It then follows using Lemma~\ref{bdresults}(b)(iii) that
\begin{equation}
\label{limpBDn}
\lim_{n \to \infty}{\rm P}\left(\tilde{B}^{(n)}(\tilde{D}^{(n)})\ge n^{\frac{1}{3}}\right)=0.
\end{equation}
Also, since $\tilde{D}^{(n)}\convp 0$ as $n \to \infty$, it follows using~\eqref{kurtzsushat} that, for any $\epsilon >0$,
\begin{equation}
\label{limbarS>s1-ep}
\lim_{n \to \infty}{\rm P}\left(\bar S^{(n)}_0(t)>\hat{s}_1-\epsilon \mbox{  for all } t \in [0, \tilde{D}^{(n)}]\right)=1.
\end{equation}

Suppose that $\tilde{B}^{(n)}(\tilde{D}^{(n)}) < n^{\frac{1}{3}}$ and, for fixed $\epsilon \in (0,\hat{s}_1)$, $\bar S^{(n)}_0(t)>\hat{s}_1-\frac{\epsilon}{2}$ for all $t \in [0, \tilde{D}^{(n)}]$.  Then, $S^{(n)}(t_i) \ge S^{(n)}_0(t_i)- n^{\frac{1}{3}}$, for $i=1,2,\cdots, \tilde{B}^{(n)}(\tilde{D}^{(n)})$, so if $p_i^{(n)}$ denotes the probability that the $i$th birth in  $\left\{\tilde{I}^{(n)}(t):t \ge 0\right\}$ yields an infection in $\left\{I^{(n)}(t):t \ge 0\right\}$, then 
\[
p_i^{(n)}=\frac{S^{(n)}(t_i)}{S^{(n)}_0(t_i)}\ge 1-\frac{n^{\frac{1}{3}}}{S^{(n)}_0(t_i)} \ge 1-\frac{n^{-\frac{2}{3}}}{\hat{s}_1-\epsilon},
\]
whence
\begin{eqnarray*}
\prod_{i=1}^{\tilde{B}^{(n)}(\tilde{D}^{(n)})} p_i^{(n)} &\ge& \left(1-\frac{n^{-\frac{2}{3}}}{\hat{s}_1-\epsilon}\right)^{\tilde{B}^{(n)}(\tilde{D}^{(n)})}\\
&\ge&\left(1-\frac{n^{-\frac{2}{3}}}{\hat{s}_1-\epsilon}\right)^{n^{\frac{1}{3}}}\\
&\ge& 1-\frac{n^{-\frac{1}{3}}}{\hat{s}_1-\epsilon}\\
&\to& 1\quad\mbox{as }n \to \infty.
\end{eqnarray*}

Thus, recalling~\eqref{limpBDn} and~\eqref{limbarS>s1-ep}, the probability that $\left\{I^{(n)}(t):t \ge 0\right\}$ and $\left\{\tilde{I}^{(n)}(t):t \ge 0\right\}$ coincide over $[0,\tilde{D}^{(n)}]$ converges to one as $n \to \infty$, which, together with~\eqref{intildeconv}, yields
\begin{equation*}
\lim_{n \to \infty}{\rm P}\left(\chi_1^{(n)}=1\right)=1={\rm P}\left(\chi_1=1\right),
\end{equation*}
and, together with~\eqref{dntildeconv}, yields
\begin{equation*}
{D}^{(n)}\convp 0 \quad \mbox{as } n \to \infty.
\end{equation*}

We have thus proved parts (i) and (iii) for $k=1$.  Note that, since ${\rm P}\left(\chi_k=0\mbox{ for all }k=1,2,\cdots\right)=0$, when reverting to the original time axis, the probability that the total number of individuals infected during $[0, w_1^{(n)}]$ in $E^{(n)}$ is less than $n^{\frac{5}{12}}$ tends to one as $n \to \infty$, which combined with~\eqref{kurtzsus0} proves part (ii) when $k=1$.  Parts (i), (ii) and (iii) for $k>1$ follow easily by induction since the processes $\left\{(S^{(n)}(t),I^{(n)}(t)):t \ge 0\right\}$ ($n=1,2,\cdots$) and $S$ are Markov.

\end{proof}

Before proceeding we state some well-known facts about the final outcome of the deterministic general epidemic (e.g.~\cite{AB00a} Chapter 1.4).  For $t \ge 0$, let $s(t)$ and $i(t)$ denote respectively the density of susceptibles and infectives at time $t$, so $(s(t),i(t))$ are determined by the differential equations
\begin{equation}
\label{gde0}
\dfrac{ds}{dt}=-R_0 s i,\qquad \dfrac{di}{dt}=R_0 s i- i,
\end{equation}
with initial condition $\left(s(0),i(0)\right)=(s_0,i_0)$, where $s(0)>0$ and $i(0)>0$.  Note that time is scaled so that the recovery rate is $1$.  Then $s(t)$ decreases with $t$, $\lim_{t \to \infty}i(t)=0$ and
$\lim_{t \to \infty}s(t)=s_{\infty}(s_0,i_0)$, where $s_{\infty}(s_0,i_0)$ is the unique solution in $(0,1)$ of
\[
s_{\infty}=s_0{\rm e}^{-R_0(s_0+i_0-s_{\infty})}.
\]
Note that $s_{\infty}$ is continuous in $(s_0,i_0)$ and $s_{\infty}(s_0,i_0) \to s_{\infty}(s_0,0)$ as $i_0 \downarrow 0$, where (recall~\eqref{tau})
\begin{equation*}
s_{\infty}(s_0,0)=\left\{
  \begin{array}{l l}
   0 & \quad \text{if } R_0 s_o \le 1,\\
   s_0(1-\tau(s_0)) & \quad \text{if } R_0 s_o > 1.
  \end{array} \right.
\end{equation*}

In the following two lemmas, there is no importation of infectives in $\left\{\left(S^{(n)}(t),I^{(n)}(t)\right):t \ge 0\right\}$, though births of susceptibles still occur at rate $\mu n (1-\kappa_n)$.  For $t \ge 0$, let $\bar I^{(n)}(t)=n^{-1}I^{(n)}(t)$.

\begin{lem}
\label{majoreps0}
Suppose that $\left(\bar S^{(n)}(0),\bar I^{(n)}(0)\right) \convp (s_0,i_0)$ as $n \to \infty$, where $s_0>\frac{1}{R_0}$ and $i_0>0$.  Let $u^{(n)}_1=\inf\{t>0:I^{(n)}(t)=0\}$.
Then, as $n \to \infty$,
\begin{enumerate}
\item[(i)] $\bar S^{(n)}(u^{(n)}_1) \convp s_{\infty}(s_0,i_0)$,
\item[(ii)] $u^{(n)}_1 \convp 0$.
\end{enumerate}
\end{lem}

\begin{proof}
For $n=1,2,\cdots$ and $t>0$, let $\tilde{S}^{(n)}(t)=S^{(n)}(t/\gamma_n)$ and $\tilde{I}^{(n)}(t)=I^{(n)}(t/\gamma_n)$.
Let $\mathbf{X}^{(n)}=\left\{\mathbf{X}^{(n)}(t): t \ge 0\right\}$, where $\mathbf{X}^{(n)}(t)=\left(\tilde{S}^{(n)}(t),\tilde{I}^{(n)}(t)\right)$.  The process $\mathbf{X}^{(n)}$ is a continuous-time Markov chain with transition intensities
\begin{align*}
q^{(n)}_{(s,i),(s+1,i)}&=n\left[\frac{(1-\kappa_n)\mu}{\gamma_n}\right],\\
q^{(n)}_{(s,i),(s-1,i)}&=n\left[\frac{\mu}{\gamma_n}\frac{s}{n}\right],\\
q^{(n)}_{(s,i),(s-1,i+1)}&=n\left[R_0\frac{s}{n}\frac{i}{n}+\left(\frac{\lambda_n}{\gamma_n}-R_0\right)\frac{s}{n}\frac{i}{n}\right],\\
q^{(n)}_{(s,i),(s,i-1)}&=n\left[\frac{i}{n}+\frac{\mu}{\gamma_n}\frac{i}{n}\right],
\end{align*}
corresponding to a birth of a susceptible, a death of a susceptible, an infection of a susceptible, and a recovery or death of an infective, respectively.

The transition intensities are written in the above form to indicate that the family of processes $\left\{\mathbf{X}^{(n)}:n=1,2,\cdots\right\}$ is asymptotically density dependent, as defined by~\cite{Pollett90}.
Let $E$ be any compact subset of $[0,\infty)^2$.  Recall that $\kappa_n \to 0, \gamma_n \to \infty$ and $\frac{\lambda_n}{\gamma_n} \to R_0$ as $n \to \infty$.  Hence, as $n \to \infty$, each of $\frac{(1-\kappa_n)\mu}{\gamma_n}$, $\sup_{(x,y)\in E}\frac{\mu}{\gamma_n}x$, $\sup_{(x,y)\in E} \left(\frac{\lambda_n}{\gamma_n}-R_0\right)xy$ and $ \sup_{(x,y)\in E}\frac{\mu}{\gamma_n}y$ converges to $0$. It follows that the conditions of Theorem 3.1 in~\cite{Pollett90} are satisfied, whence,
for any $\epsilon>0$ and any $t>0$,
\begin{equation}
\label{kurtzepi}
\lim_{n \to \infty}{\rm P}\left(\sup_{0\le u \le t}\left|\frac{1}{n}\mathbf{X}^{(n)}(t)-\mathbf{x}(t)\right|<\epsilon\right)=1,
\end{equation}
where $\mathbf{x}(t)=\left(s(t),i(t)\right)$ is the solution of the deterministic general epidemic~\eqref{gde0} having initial condition
$\left(s(0),i(0)\right)=(s_0,i_0)$.
Write $s_\infty$ for $s_{\infty}(s_0,i_0)$.  There exists $\epsilon_0>0$ such that $R_0(s_\infty+\epsilon_0)<1$, since otherwise $\lim_{t \to \infty} i(t)$ would be strictly positive.  Given $\epsilon \in (0, \epsilon_0)$, choose $\epsilon'>0$ so that
\begin{equation}
\label{epsilon'}
\epsilon'\frac{R_0(s_\infty+\epsilon_0)}{1-R_0(s_\infty+\epsilon_0)}<\frac{\epsilon}{8}.
\end{equation}
There exists $t_1>0$ such that $ i(t_1)<\epsilon'$ and $ s(t_1)\in[s_\infty,s_\infty+\frac{\epsilon}{3})$. Then~\eqref{kurtzepi} implies that
\[
\lim_{n \to \infty}{\rm P}\left(\left|\frac{1}{n} \tilde{S}^{(n)}(t_1)-s_\infty\right|<\frac{\epsilon}{2}\right)=1
\quad\mbox{and}\quad
\lim_{n \to \infty}{\rm P}\left(\frac{1}{n} \tilde{I}^{(n)}(t_1)<\frac{3}{2}\epsilon'\right)=1,
\]
so, reverting to the original time scale and letting $t_n=t_1/\gamma_n$,
\begin{equation}
\label{sitnbounds}
\lim_{n \to \infty}{\rm P}\left(\left|\bar{S}^{(n)}(t_n)-s_\infty\right|<\frac{\epsilon}{2}\right)=1
\quad\mbox{and}\quad
\lim_{n \to \infty}{\rm P}\left(\bar {I}^{(n)}(t_n)<\frac{3}{2}\epsilon'\right)=1.
\end{equation}

Observe that, whilst $\bar S^{(n)}(t_n+t) \le s_\infty+\epsilon$, the process $\left\{I^{(n)}(t_n+t):t \ge 0\right\}$ is bounded above by the birth-and-death process $Z_{\tilde{\alpha}_n,\beta_n,\left \lceil{\frac{3}{2}\epsilon' n} \right \rceil}$, where $\tilde{\alpha}_n=(s_\infty+\epsilon)\lambda_n$ and $\beta_n=\gamma_n+\mu$.  Now
$\tilde{\alpha}_n/\beta_n \to R_0(s_\infty+\epsilon)$ as $n \to \infty$, so, for all sufficiently large $n$, $Z_{\tilde{\alpha}_n,\beta_n,\left \lceil{\frac{3}{2}\epsilon' n} \right \rceil}$ is in turn bounded above
by $Z_{\alpha_n,\beta_n,\left \lceil{\frac{3}{2}\epsilon' n} \right \rceil}$, where $\alpha_n= R_0(s_\infty+\epsilon_0)\beta_n$.

Recall that $B_{\alpha,\beta,k}$ and $\tau_{\alpha,\beta,k}(0)$ denote the total number of births in and the extinction time of $Z_{\alpha,\beta,k}$, respectively.  Then
\[
{\rm E}\left[B_{\alpha_n,\beta_n,1}\right]=\frac{R_0(s_\infty+\epsilon_0)}{1-R_0(s_\infty+\epsilon_0)},
\]
and, recalling~\eqref{epsilon'}, application of the strong law of large numbers yields
\begin{equation}
\label{bbounds}
\lim_{n \to \infty} {\rm P}\left(\frac{1}{n}B_{\alpha_n,\beta_n,\left \lceil{\frac{3}{2}\epsilon' n}\right \rceil} < \frac{\epsilon}{4}\right)=1.
\end{equation}
Also, Lemma~\ref{bdresults}(a)(iii) implies that
\begin{equation}
\label{bdext}
\tau_{\alpha_n,\beta_n,\left \lceil{\frac{3}{2}\epsilon' n}\right \rceil}(0) \convp 0 \quad\mbox{as }  n \to \infty.
\end{equation}

Recall that $\left\{S^{(n)}_0(t):t \ge 0\right\}$ denotes the process that describes the number of susceptibles in the absence of any infectives
and suppose that $S^{(n)}_0(0)=S^{(n)}(t_n)$.
For $t \ge 0$, let $B^{(n)}_0(t)$ and $D^{(n)}_0(t)$ be the total number of births and deaths, respectively, during $(0,t]$ in $\left\{ S^{(n)}_0(t):t \ge 0\right\}$.  Using~\eqref{kurtzsus0} and the fact that $B^{(n)}_0(t)$ has a Poisson distribution with mean $n \mu t$, there exists $\hat{t}=\hat{t}(\epsilon)>0$ such that
\begin{equation}
\label{bdsbounds0}
\lim_{n \to \infty} {\rm P}\left(B_0^{(n)}(\hat{t})<\frac{n \epsilon}{4}\right)=1 \quad\mbox{and}\quad\lim_{n \to \infty} {\rm P}\left(D_0^{(n)}(\hat{t})<\frac{n \epsilon}{4}\right)=1.
\end{equation}

The processes $\left\{\left(S^{(n)}(t_n+t),I^{(n)}(t_n+t)\right):t \ge 0\right\}$ and $Z_{\alpha_n,\beta_n,\left \lceil{\frac{3}{2}\epsilon' n} \right \rceil}$ can be coupled so that $I^{(n)}(t_n+t) \le Z_{\alpha_n,\beta_n,\left \lceil{\frac{3}{2}\epsilon' n} \right \rceil}(t)$
whilst $\bar S^{(n)}(t_n+t) \le s_\infty +\epsilon$.  The first equations in~\eqref{sitnbounds} and ~\eqref{bdsbounds0} imply that
\[
\lim_{n \to \infty}{\rm P}\left(\sup_{0 \le t \le \hat{t}}\bar S^{(n)}(t_n+t)\le s_\infty +\epsilon\right)=1,
\]
so~\eqref{bdext} implies that, with probability tending to $1$ as $n \to \infty$, the coupling holds thoughout the lifetime of
$Z_{\alpha_n,\beta_n,\left \lceil{\frac{3}{2}\epsilon' n} \right \rceil}$.  Recall that $u^{(n)}_1=\inf\{t>0:I^{(n)}(t)=0\}$.  The
coupling implies that $u^{(n)}_1-t_n \le \tau_{\alpha_n,\beta_n,\left \lceil{\frac{3}{2}\epsilon' n}\right \rceil}$, so part (ii) of the lemma follows from~\eqref{bdext}, since $t_n \to 0$ as $n \to \infty$. Further, $S^{(n)}(u^{(n)}_1)$ is at most the sum of $S^{(n)}(t_n)$ and the number of births in $(t_n,u^{(n)}_1]$, and at least the difference between $S^{(n)}(t_n)$ and the sum of
the number of susceptible deaths in $(t_n,u^{(n)}_1]$ and $B_{\alpha_n,\beta_n,\left \lceil{\frac{3}{2}\epsilon' n}\right \rceil}$,
so~\eqref{sitnbounds},~\eqref{bbounds} and~\eqref{bdsbounds0} imply that
\[
\lim_{n \to \infty}{\rm P}\left(s_\infty-\epsilon<\bar S^{(n)}(u^{(n)}_1)<s_\infty+\frac{3}{4}\epsilon\right)=1,
\]
proving part (i) of the lemma, since $\epsilon \in (0, \epsilon_0)$ can be arbitrarily small.
\end{proof}

\begin{lem}
\label{majoreps}
Suppose that $I^{(n)}(0)=\left \lceil{\log n} \right \rceil$ $(n=1,2,\cdots)$ and \newline $\bar S^{(n)}(0) \convp s_0$ as $n \to \infty$, where $s_0>\frac{1}{R_0}$.  Let $u^{(n)}_1=\inf\left\{t>0:I^{(n)}(t)=0\right\}$,
\[
c^{(n)}_1=\min_{0 \le t \le u^{(n)}_1} \bar S^{(n)}(t) \quad \mbox{and} \quad \tilde{c}^{(n)}_1=\max_{0 \le t \le u^{(n)}_1} \bar S^{(n)}(t).
\]
Then, as $n \to \infty$,
\begin{enumerate}
\item[(i)] $\bar S^{(n)}(u^{(n)}_1) \convp s_0(1-\tau(s_0))$, where the function $\tau(s)$ is defined at~\eqref{tau};
\item[(ii)] $u^{(n)}_1 \convp 0$;
\item[(iii)] $c^{(n)}_1 \convp s_0(1-\tau(s_0))$ and $\tilde{c}^{(n)}_1 \convp s_0$.
\end{enumerate}
\end{lem}

\begin{proof}
Fix $\theta \in (0,1)$ such that $s_0(1-3\theta)R_0>1$.  Then, whilst $\bar S^{(n)}(t) \ge s_0(1-2\theta)$, $\{I^{(n)}(t):t \ge 0\}$ is bounded below by the birth-and-death process $Z_{\tilde{\alpha}_n(\theta),\beta_n,\left \lceil{\log n} \right \rceil}$, where $\tilde{\alpha}_n(\theta)=\lambda_n s_0 (1-2\theta)$ and $\beta_n=\gamma_n+\mu$.  Now
$\tilde{\alpha}_n(\theta)/\beta_n \to R_0s_0(1-2\theta)$ as $n \to \infty$, so, for all sufficiently large $n$, $Z_{\tilde{\alpha}_n(\theta),\beta_n,\left \lceil{\log n} \right \rceil}$ is in turn bounded below by $Z_{\alpha_n(\theta),\beta_n,\left \lceil{\log n} \right \rceil}$, where $\alpha_n(\theta)=R_0 s_0 (1-3\theta)\beta_n$.

Recall that $E^{(n)}$ denotes the epidemic process indexed by $n$. For $t>0$, let $B^{(n)}(t)$ be the total number of infections in $E^{(n)}$ during $(0,t]$.  Let $\tau^{(n)}_\theta=\inf\left\{t>0:B^{(n)}(t) \ge \theta s_0 n\right\}$.   Define $\left\{\bar S^{(n)}_0(t):t \ge 0\right\}$ as in the proof of Lemma~\ref{minoreps}.  For $t \ge 0$, let $B^{(n)}_0(t)$ and $D^{(n)}_0(t)$ be the total number of births and deaths, respectively, during $(0,t]$ in $\left\{S^{(n)}_0(t):t \ge 0\right\}$.  As at~\eqref{bdsbounds0}, but note that $\bar S^{(n)}_0(0)$ is different here,
for any $\epsilon >0$ there exists $\hat{t}(\epsilon)>0$ such that
\begin{equation}
\label{bdsbounds}
\lim_{n \to \infty} {\rm P}\left(B_0^{(n)}(\hat{t}(\epsilon))<\frac{n \epsilon}{4}\right)=1 \qquad \mbox{and} \qquad\lim_{n \to \infty} {\rm P}\left(D_0^{(n)}(\hat{t}(\epsilon))<\frac{n \epsilon}{4}\right)=1.
\end{equation}
Also, since $\bar S^{(n)}(0) \convp s_0$ as $n \to \infty$,
\begin{equation}
\label{s0convp}
\lim_{n \to \infty} {\rm P}\left(|\bar S^{(n)}(0)-s_0|<\frac{\epsilon}{2}\right)=1.
\end{equation}

Observe that, if $\tau^{(n)}_\theta \le \hat{t}(\epsilon), \left|\bar S^{(n)}(0)-s_0\right|<\frac{\epsilon}{2},B_0^{(n)}(\hat{t}(\epsilon))<\frac{n \epsilon}{4}$ and $D_0^{(n)}(\hat{t}(\epsilon))<\frac{n \epsilon}{4}$, then 
\begin{equation}
\label{susupper}
\max_{0 \le t \le \tau^{(n)}_\theta} \bar S^{(n)}(t) \le s_0(1-\theta)+\frac{3}{4}\epsilon,
\end{equation}
obtained by making $\bar S^{(n)}(0)$ and $B_0^{(n)}(\hat{t}(\epsilon))$ as large as possible and assuming no susceptible dies during $[0,\tau^{(n)}_\theta]$, and
\begin{equation}
\label{suslower}
\min_{0 \le t \le \tau^{(n)}_\theta} \bar S^{(n)}(t) \ge s_0-\frac{3}{4}\epsilon-n^{-1}B^{(n)}(\tau^{(n)}_\theta),
\end{equation}
obtained by making $\bar S^{(n)}(0)$ as small as possible, $D_0^{(n)}(\hat{t}(\epsilon))$ as large as possible and assuming no susceptible is born during $[0,\tau^{(n)}_\theta]$.

Recall that, whilst $\bar S^{(n)}(t) \ge s_0(1-2\theta)$, $\left\{I^{(n)}(t):t \ge 0\right\}$ is bounded below by the birth-and-death process $Z_{\alpha_n(\theta),\beta_n,\left \lceil{\log n} \right \rceil}$, so $\tau^{(n)}_\theta \le
\hat{\tau}^{(n)}_{\alpha_n,\beta_n,\left \lceil{\log n} \right \rceil}(s_0\theta n)$, provided $\bar S^{(n)}(t) \ge s_0(1-2\theta)$ throughout $[0,\tau^{(n)}_\theta]$.  Now $\hat{\tau}^{(n)}_{\alpha_n,\beta_n,\left \lceil{\log n} \right \rceil}(s_0\theta n) \convp 0$ as  $n \to \infty$, by Lemma~\ref{bdresults}(b)(iv), so ${\rm P}\left(\hat{\tau}^{(n)}_{\alpha_n,\beta_n,\left \lceil{\log n} \right \rceil}(s_0\theta n) < \hat{t}(\epsilon)\right) \to 1$ as $n \to \infty$, for any $\epsilon>0$. Setting $\epsilon=s_0 \theta$ in~\eqref{suslower}, using~\eqref{bdsbounds},~\eqref{s0convp} and noting that $n^{-1}B^{(n)}(\tau^{(n)}_\theta) \convp s_0\theta$ as n $ \to \infty$, shows that
\[
\lim_{n \to \infty} {\rm P}\left(\min_{0 \le t \le \tau^{(n)}_\theta} \bar S^{(n)}(t) \ge  s_0(1-2\theta)\right)=1,
\]
so
\begin{equation}
\label{tauthetanconv}
\tau^{(n)}_\theta \convp 0 \quad\mbox{as } n \to \infty.
\end{equation}
Further, since for any $\epsilon > 0$, ${\rm P}\left(\tau^{(n)}_\theta < \hat{t}(\epsilon)\right) \to 1$ as $n \to \infty$, it follows from~\eqref{bdsbounds}-\eqref{suslower} that, for any $\epsilon>0$,
\begin{equation*}
\lim_{n \to \infty}{\rm P}\left(s_0(1-\theta)-\epsilon < \bar S^{(n)}(\tau^{(n)}_\theta)<s_0(1-\theta)+\epsilon\right)=1,
\end{equation*}
so
\begin{equation}
\label{Sntauconvp}
\bar S^{(n)}(\tau^{(n)}_\theta) \convp s_0(1-\theta) \quad \mbox{as } n \to \infty.
\end{equation}

It is straightforward to couple the jump processes of $\left\{(I^{(n)}(t),B^{(n)}(t)): t \ge 0\right\}$ and\\ $\left\{Z_{\alpha_n(\theta),\beta_n,\left \lceil{\log n} \right \rceil}(t),B_{\alpha_n(\theta),\beta_n,\left \lceil{\log n} \right \rceil}(t):t \ge 0\right\}$ to show that $I^{(n)}(\tau^{(n)}_\theta)\sge
Z_{\alpha_n(\theta),\beta_n,\left \lceil{\log n} \right \rceil}(\hat{\tau}^{(n)}_\theta)$, where $\hat{\tau}^{(n)}_\theta=\hat{\tau}^{(n)}_{\alpha_n,\beta_n,\left \lceil{\log n} \right \rceil}(s_0\theta n)$
and $\sge$ denotes stochastically greater than.  
Further, recalling that $Z_{\alpha_n(\theta),\beta_n,\left \lceil{\log n} \right \rceil}$ has the same distribution as $\left\{Z_{R_0(1-3\theta),1,\left \lceil{\log n} \right \rceil}(\beta_n t):t \ge 0\right\}$, it follows using~\cite{Nerman81},\newline Theorem 5.4, that
\[
\frac{Z_{\alpha_n(\theta),\beta_n,\left \lceil{\log n} \right \rceil}(\hat{\tau}^{(n)}_\theta)}{B_{\alpha_n(\theta),\beta_n,\left \lceil{\log n} \right \rceil}(\hat{\tau}^{(n)}_\theta)} \convas 1-\frac{1}{R_0 s_0(1-3\theta)}\quad \mbox{as } n \to \infty.
\]
Thus, since $n^{-1} B_{\alpha_n(\theta),\beta_n,\left \lceil{\log n} \right \rceil}(\hat{\tau}^{(n)}_\theta)\convp s_0 \theta$ as $n \to \infty$,
\begin{equation}
\label{inflower}
\lim_{n \to \infty} {\rm P}\left(\bar I^{(n)}(\tau^{(n)}_\theta) > i_-(\theta)\right)=1,
\end{equation}
where
\[
i_-(\theta)=\left[1-\frac{2}{R_0 s_0(1-3\theta)}\right]s_0\theta.
\]
A similar argument using an upper bounding birth-and-death process yields that
\begin{equation}
\label{infupper}
\lim_{n \to \infty} {\rm P}\left(\bar I^{(n)}(\tau^{(n)}_\theta) <i_+(\theta)\right)=1,
\end{equation}
where
\[
i_+(\theta)=\left[1+\frac{2}{R_0 s_0(1-3\theta)}\right]s_0\theta.
\]

Exploiting the Markov property of $\left\{\left(S^{(n)}(t),I^{(n)}(t)\right):t \ge 0\right\}$, \eqref{Sntauconvp}-\eqref{infupper}
and Lemma~\ref{majoreps0}(i) imply that, for any $\epsilon>0$,
\begin{equation*}
\lim_{n \to \infty}{\rm P}\left(s_\infty(s_0(1-\theta),i_-(\theta))-\epsilon<\bar S^{(n)}(u^{(n)}_1)<s_\infty(s_0(1-\theta),i_+(\theta))+\epsilon\right)=1.
\end{equation*}
Letting $\theta \downarrow 0$, noting that $i_-(0+)=i_+(0+)=0$ and using the continuity properties of $s_\infty$, yield
that, for any $\epsilon>0$,
\[
\lim_{n \to \infty}{\rm P}\left(\left|\bar S^{(n)}(u^{(n)}_1)-s_0(1-\tau(s_0))\right|<\epsilon\right)=1,
\]
proving part (i) of the lemma.  Part (ii) follows immediately using~\eqref{tauthetanconv},\newline ~\eqref{Sntauconvp},~\eqref{infupper}
and Lemma~\ref{majoreps0}(ii).  Part (iii) is an easy concequence of parts (i) and (ii) and~\eqref{bdsbounds}.
\end{proof}

\begin{proof}[Proof of Lemma~\ref{convd}]
The lemma follows easily by induction using Lemmas~\ref{minoreps} and Lemmas~\ref{majoreps}. First note Lemma~\ref{minoreps} (i) and (iii) imply that $t^{(n)}_1|\eta \convD t_1$ as $n \to \infty$, and Lemma~\ref{minoreps} (i) and (ii) imply that~\eqref{susboundcd}
holds for $k=0$ and $\bar S^{(n)}(t^{(n)}_1-)|\eta \convp \tilde{s}_1$ as $n \to \infty$. Lemma~\ref{majoreps} (ii) then yields that
$u^{(n)}_1|\eta \convD t_1$ as $n \to \infty$, Lemma~\ref{majoreps} (i) yields that $s^{(n)}_1|\eta \convD s_1$ as $n \to \infty$,  and Lemma~\ref{majoreps} (iii) yields that $c^{(n)}_1 |\eta \convD s_1$ and $\tilde{c}^{(n)}_1 |\eta \convD \tilde{s}_1$ as $n \to \infty$.  Now $\left\{\left(S^{(n)}(t),I^{(n)}(t)\right):t \ge 0\right\}$ is Markov, so, since $\bar S^{(n)}(u^{(n)}_1)|\eta \convp s_1$ as $n \to \infty$, the above argument can be repeated for $k=2,3,\cdots$.  Part (iii) is immediate, since $\{t_1,t_2,\cdots\} \subseteq \{r_1,r_2,\cdots\}$, where $r_1,r_2,\cdots$ are the times of the points in $\eta$.
\end{proof}

\section{Discussion}\label{disc}
In the paper it is  proved that for an SIR epidemic in a dynamic population (whose size fluctuates around $n$), in which there is importation of infectives at a constant rate, the normalised process of susceptibles converges to a regenerative process $S$ as $n\to\infty$. Further, properties of the limiting process $S$ are derived.  The asymptotic regime considered is for the situation when the rate of importation of infectives $\kappa \mu$ and the basic reproduction number $R_0$ remain constant with $n$, whereas the average length of the infectious period $1/\gamma_n$ converges to 0 faster than $1/\log n$ (in most real-life epidemics, the ratio of average infectious period and average lifetime lies between $10^{-4}$ and $10^{-3}$).

Other asymptotic regimes could of course also be considered. For example, if the importation rate of infectives grows with $n$, then there will always be infectives present in the population resembling an endemic situation. If the duration of an infectious period remains fixed (or at least grows slower than $\log n$), then the duration of a single outbreak will be long and the typical time horizon will not go beyond the first outbreak. A more complicated and interesting scenario seems to be for the asymptotic situation treated in the current paper, but where the epidemic is initiated with a fraction $1/R_0$ of the population susceptible and a large enough number of infectives. It then seems as if an endemic equilibrium will stabilize, but determining and proving this rigorously remains an open problem. For large but finite $n$, it is possible for the process to get stuck in an endemic situation near the end of a major outbreak (with states similar to those just described). Eventually the epidemic leaves this endemic state and returns to the behaviour of the limiting process. In Figure~\ref{endemic-simul} such a simulation is presented. The parameter values are $n=100,000$, $\mu=1/75$, $\kappa=1$ (so the importation rate of infectives is one per $75$ years),
$R_0=2$ and $\gamma=2$ (so the average infectious period is $6$ months).  The left and right plots show the fraction of the population that are susceptible and infective, respectively, as functions of time.  A quasi-endemic phase lasts roughly from years $1,300$ to $3,000$.  Observe that major outbreaks become smaller prior to the process entering the quasi-endemic phase and fluctuations in the number of infectives increase in amplitude prior to the end of the quasi-endemic phase.
\begin{figure}
\begin{center}
\resizebox{\hfigwidth}{!}{\includegraphics{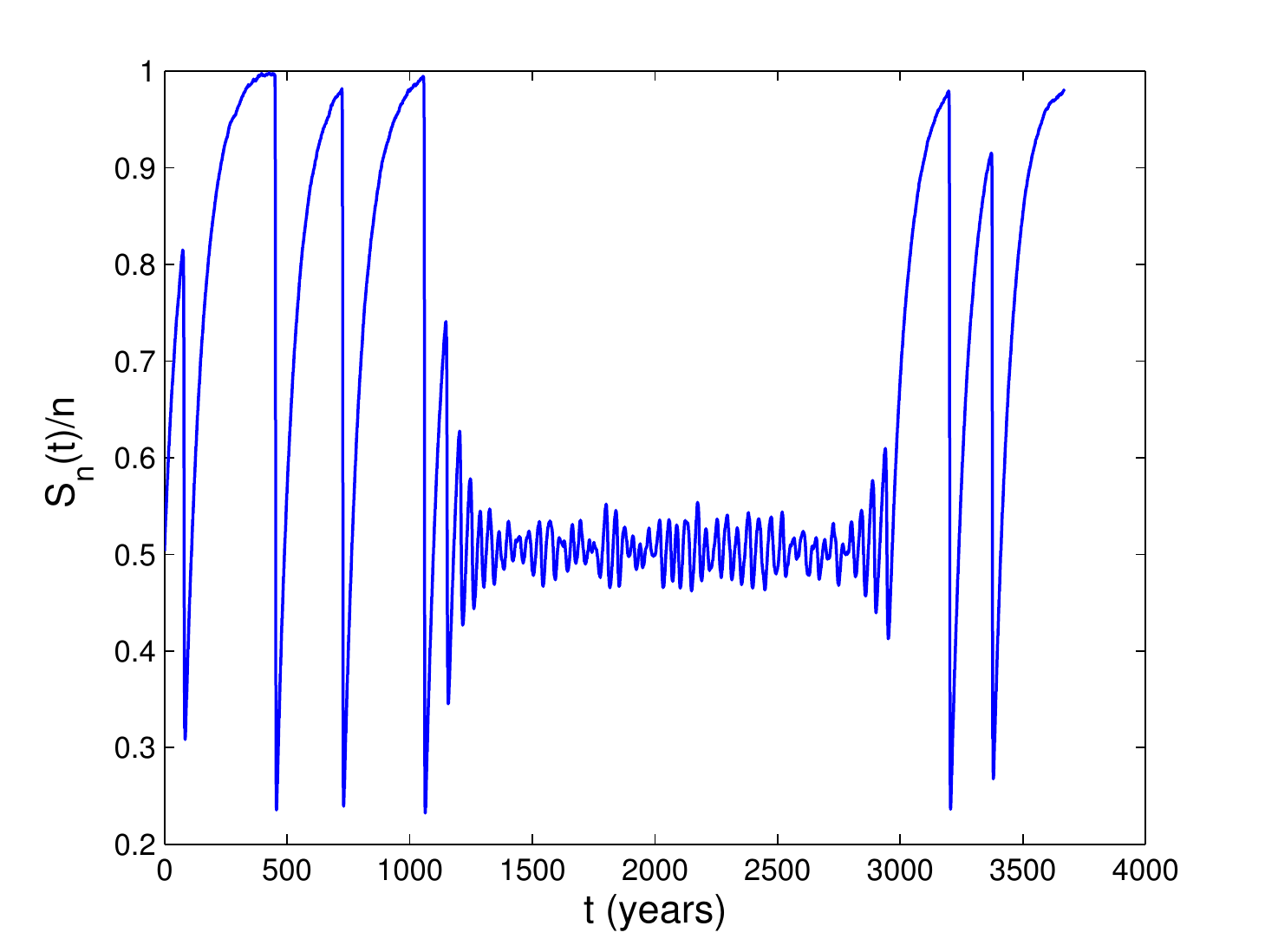}}
\resizebox{\hfigwidth}{!}{\includegraphics{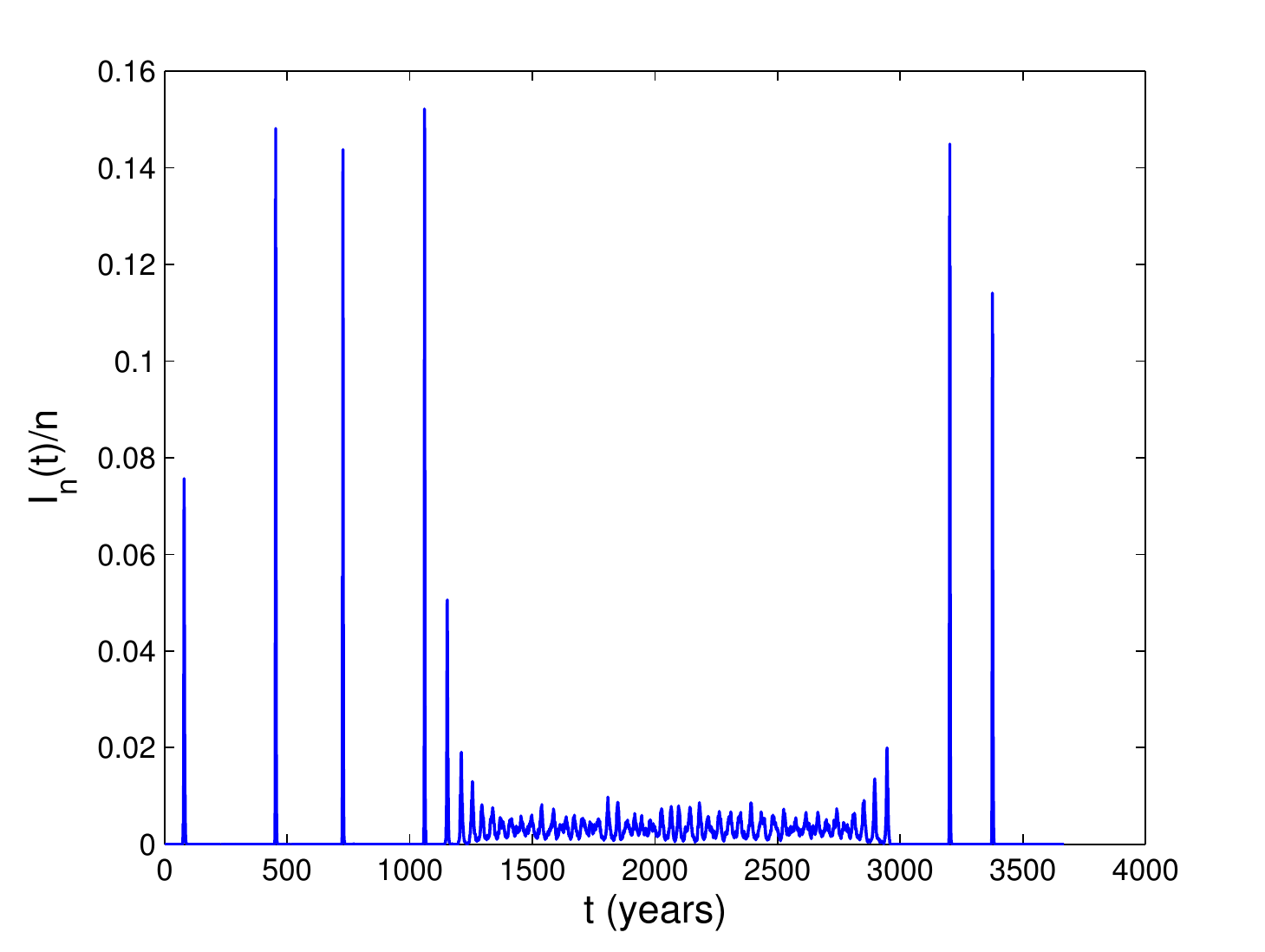}}
\end{center}
\caption{Plot of an epidemic exhibiting  quasi-endemic behaviour. }
\label{endemic-simul}
\end{figure}
Beside studying other asymptotic regimes, it could be of interest to increase realism in the model, for example, by relaxing exponential distributions of infectious periods and lifetimes and allowing for a latent state (cf.~\cite{AB00b}, who consider epidemics with importation of susceptibles only) or by having some population structure, such as network or households (see the challenges in~\cite{PBBEHIT15} and ~\cite{BBHIMPS15}).

\appendix
\section{Proof of Lemma~\ref{bdresults}}

\label{appendix1}

Suppose that $a<1$.  Then, for $t>0$,
\begin{eqnarray*}
\lim_{n \to \infty}{\rm P}\left(\tau_{\alpha_n,\beta_n,1}(0)>t\right)&=&
\lim_{n \to \infty}{\rm P}\left(Z_{\alpha_n,\beta_n,1}(t)\ge1\right)\\
&\le&\lim_{n \to \infty}{\rm E}\left[Z_{\alpha_n,\beta_n,1}(t)\right]\\
&=&\exp\left(-(1-a)\beta_n t\right) \to 0 \quad \mbox{as } n\to \infty,
\end{eqnarray*}
since $\beta_n \to \infty$ as $n \to \infty$, proving part (a)(i).

Observe that, for any $k$, $\{Z_{\alpha_n,\beta_n,k}(t):t\ge 0\}\eqD \{Z_{a,1,k}(\beta_n t):t\ge 0\}$, where $\eqD$ denotes equal in distribution.  It follows that $\tau_{\alpha_n,\beta_n,1}(\log n) \eqD \frac{1}{\beta_n} \tau_{a,1,1}(\log n)$, so
\begin{eqnarray*}
\lim_{n\to\infty}{\rm P}\left(\tau_{\alpha_n,\beta_n,1}(\log n)= \infty\right)&=&
\lim_{n\to\infty}{\rm P}\left(\tau_{a,1,1}(\log n)= \infty\right)\\
&=&1,
\end{eqnarray*}
since $Z_{a,1,1}$ is subcritical, proving part (a)(ii).

For any $t>0$,
\begin{eqnarray*}
{\rm P}\left(\tau_{\alpha_n,\beta_n,\left \lceil{c n}\right \rceil}>t\right)&=&
{\rm P}\left(Z_{\alpha_n,\beta_n,\left \lceil{c n}\right \rceil}(t)\ge 1 \right)\\
&\le&{\rm E}\left[Z_{\alpha_n,\beta_n,\left \lceil{c n}\right \rceil}(t)\right]\\
&=&\left \lceil{c n}\right \rceil \exp\left(-(1-a)\beta_n t\right)\\
&\to&0\quad\mbox{as } n \to \infty,
\end{eqnarray*}
since $\log n/\beta_n \to 0$ as $n \to \infty$, proving part (a)(iii).

Suppose that $a>1$.  Note that, since $\{Z_{\alpha_n,\beta_n,1}(t):t\ge 0\}\eqD \{Z_{a,1,1}(\beta_n t):t\ge 0\}$,
coupled realisations of $\{Z_{\alpha_n,\beta_n,1}:n=1,2,\cdots\}$ can be obtained by setting $Z_{\alpha_n,\beta_n,1}(t)=Z_{a,1,1}(\beta_n t)$ $(n=1,2,\cdots;t \ge 0)$.  Now, see e.g.~\cite{AN72}, page 112, there exists a random variable $W\ge 0$, satisfying $W=0$ if and only if $Z_{a,1,1}$ goes extinct, such that
\begin{equation}
\label{bdconv0}
{\rm e}^{-(a-1)t}Z_{a,1,1}(t) \convas W \quad \mbox{as } t \to \infty.
\end{equation}
Note that, for any $x \ge 0$, $\tau_{\alpha_n,\beta_n,1}(x)=\beta_n^{-1}\tau_{a,1,1}(x)$. If $W=0$, then $\tau_{a,1,1}(0)<\infty$, so $\tau_{\alpha_n,\beta_n,1}(0)\downarrow 0$ as $n \to \infty$, and
$\max_{t\ge 0} Z_{a,1,1}(t)<\infty$, so $\tau_{a,1,1}(x)=\infty$ for all sufficiently large $x$, whence $\tau_{\alpha_n,\beta_n,1}(\log n)=\infty$ for all sufficiently large $n$.  If $W>0$, then $\tau_{a,1,1}(0)=\infty$, so $\tau_{\alpha_n,\beta_n,1}(0)=\infty$.  Also, for any $t>0$,
\begin{eqnarray*}
Z_{\alpha_n,\beta_n,1}(t)\ge \log n &\iff& Z_{a,1,1}(\beta_n t) \ge \log n\\
&\iff& {\rm e}^{-(a-1) \beta _n t}Z_{a,1,1}(\beta_n t) \ge {\rm e}^{-(a-1) \beta _n t} \log n.
\end{eqnarray*}
Now ${\rm e}^{-(a-1) \beta _n t} \log n \to 0$ as $n \to \infty$, since $\lim_{n \to \infty} \log n /\beta_n=0$, so, since $W>0$,~\eqref{bdconv0} implies that $Z_{\alpha_n,\beta_n,1}(t)>\log n$ for all sufficiently large $n$.  This holds for any $t>0$, so $\tau_{\alpha_n,\beta_n,1}(\log n) \to 0$ as $n \to \infty$.  Part (b)(i) follows since ${\rm P}(W=0)=\frac{1}{a}$ and part (b)(ii) also follows, indeed we have shown that, under the coupling, $\min\left(\tau_{\alpha_n,\beta_n,1}(\log n), \tau_{\alpha_n,\beta_n,1}(0)\right)\convas 0$ as $n \to \infty$.

Further, it follows using~\cite{Nerman81}, Theorem 5.4, that
\begin{equation}
\label{bdconvb}
{\rm e}^{-(a-1)t}B_{a,1,1}(t) \convas \frac{a}{a-1}W \quad \mbox{as } t \to \infty,
\end{equation}
where $W$ is the same random variable as in~\eqref{bdconv0}. If $W=0$, then $B_{a,1,1}(\infty)<\infty$, so $B_{\alpha_n,\beta_n,1}(\tau_{\alpha_n,\beta_n,1}(0))=B_{a,1,1}(\infty)<n^{\frac{1}{3}}$ for all sufficiently large $n$.  If $W>0$, then~\eqref{bdconv0} and~\eqref{bdconvb} imply that $\lim_{t \to \infty}B_{a,1,1}(t)/Z_{a,1,1}(t) =\frac{a}{a-1}$, so, since $\tau_{a,1,1}(\log n) \to \infty$ as $n \to \infty$, $\lim_{n \to \infty}B_{\alpha_n,\beta_n,1}(\tau_{\alpha_n,\beta_n,1}(\log n))/\log n = \frac{a}{a-1}$, whence $B_{\alpha_n,\beta_n,1}(\tau_{\alpha_n,\beta_n,1}(\log n))<n^{\frac{1}{3}}$ for all sufficiently large $n$.  Part (b)(iii) now follows.

A similar argument to the above shows that, if $W>0$, then for any $c>0$, $\tau_{\alpha_n,\beta_n,1}(cn)\to 0$ as $n \to \infty$.  Thus $\tau_{\alpha_n,\beta_n,\left \lceil{\log n}\right \rceil}(cn)\convp 0$ as $n \to \infty$, since
$Z_{\alpha_n,\beta_n,\left \lceil{\log n}\right \rceil}$ is the sum of $\left \lceil{\log n}\right \rceil$ independent copies of $Z_{\alpha_n,\beta_n,1}$.  Part (b)(iv) follows since
$\tau_{\alpha_n,\beta_n,\left \lceil{\log n}\right \rceil}(cn+1)\sge \hat{\tau}_{\alpha_n,\beta_n,\left \lceil{\log n}\right \rceil}(cn)$.

\section*{Acknowledgements}


T.~B.~is supported by Riksbankens jubilieumsfond, grant P12-0705:1
and P.~T.~is supported by Vetenskapsr{\aa}det (Swedish Research Council), project 20105873\\
F.~B.~and P.~T.~gratefully acknowledge
support from the Isaac Newton Institute for Mathematical Sciences,
Cambridge, where we held Visiting Fellowships under the Infectious Disease
Dynamics programme and its follow-up meeting, during which this work was initiated.


\begin{thebibliography}{1}


\bibitem[Andersson and Britton(2000a)]{AB00a}
Andersson, H. and Britton, T.~(2000a)
{\it Stochastic Epidemic Models and Their Statistical Analysis}.
Springer, New York.

\bibitem[Andersson and Britton(2000b)]{AB00b}
Andersson, H. and Britton, T.~(2000b)
Stochastic epidemics in dynamic populations: quasi-stationarity and extinction.
{\sl J.~Math.~Biol.}~{\bf 41}, 559--580.

\bibitem[Asmussen(1987)]{Asmussen87}
Asmussen, S.~(1987)
{\it Applied Probability and Queues}. Wiley, Chichester.

\bibitem[Athreya and Ney(1972)]{AN72}
Athreya, K.B. and Ney, P.E.~(1972)
{\it Branching Processes}. Springer-Verlag, Berlin.

\bibitem[Ball {\it et al.}~(2015)]{BBHIMPS15}
Ball, F., Britton, T., House, T., Isham, V., Mollison, D, Pellis, L. and Scalia Tomba, G.~(2015)
Seven challenges for metapopulation models of epidemics including households models.
{\sl Epidemics} {\bf 10}, 63--67.


\bibitem[Ball {\it et al.}~(1997)]{BMS-T97}
Ball, F., Mollison, D. and Scalia-Tomba, G.~(1997) Epidemics with two levels of mixing.
{\sl Ann.~Appl.~Probab.} {\bf 7}, 46--89.

\bibitem[Barbour(1997)]{Barbour75}
Barbour, A.D.~(1975) The duration of the closed stochastic epidemic.
{\sl Biometrika} {\bf 62}, 477--482.

\bibitem[Bartlett(1956)]{Bart56}
Bartlett, M.S.~(1956) Deterministic and stochastic models for recurrent epidemics.
In {\it Proc.~3rd Berkeley Symp. Mathematical Statistics and Probability} {\bf 4}, 81--109.

\bibitem[Billingsley(1968)]{Bill}
Billingsley, P.~(1968) {\it Convergence of Probability Measures.}
Wiley, New York.

\bibitem[Diekmann {\it et al.}~(2013)]{DHB13}
Diekmann, O., Heesterbeek, H. and Britton, T.~(2013)
{\it Mathematical Tools for Understanding Infectious Disease Dynamics}.  Princeton University Press,
Princeton and Oxford.

\bibitem[Ethier and Kurtz(1986)]{EK86}
Ethier, S.N. and Kurtz, T.G.~(1986)
{\it Markov Processes: Characterization and Convergence}.
Wiley, New York.


\bibitem[Hamer(1906)]{Hamer}
Hamer, W.H.~(1906) Epidemic disease in England---the evidence of variability and of persistence of type.
{\sl Lancet} {\bf 1}, 733-739.

\bibitem[N{\aa}sell(1999)]{Nasell99}
N{\aa}sell, I.~(1999) On the time to extinction in recurrent epidemics.
{\sl J.~R.~Statist.~Soc.~}B {\bf 61}, 309--330.

\bibitem[Nerman(1981)]{Nerman81}
Nerman, O.~(1981) On the convergence of supercritical general (C-M-J)
branching processes.
{\sl Z. Wahrscheinlichkeitsth} {\bf 57}, 365--395.

\bibitem[Newman(2002)]{Newman02}
Newman, M.~(2002) The spread of epidemic disease on networks. {\sl
Phys.~Rev.~E} {\bf 66}, 016128.

\bibitem[Pellis {\it et al.}~(2015)]{PBBEHIT15}
Pellis, L., Ball, F., Bansal, S., Eames, K., House, T., Isham, V. and Trapman, P..~(2015)
Eight challenges for network epidemic models.
{\sl Epidemics} {\bf 10}, 58--62.

\bibitem[Pollard(1984)]{Pollard84}
Pollard, D.~(1984) {\it Convergence of Stochastic Processes.}
Springer-Verlag, New York.

\bibitem[Pollett(1990)]{Pollett90}
Pollett P.K.~(1990) On a model for interference between searching
insect parasites.
{\sl J.~Austral.~Math.~Soc.~Ser.~B} {\bf 32}, 133--150.

\bibitem[Soper (1929)]{Soper}
Soper, H.E.~(1929) The interpretation of periodicity in disease prevalence (with discussion).
{\sl J.~R.~Statist.~Soc.} {\bf 92}, 34--73.

\bibitem[Whittle(1955)]{Whittle55}
Whittle, P.~(1955) The outcome of a stochastic epidemic - a note on Bailey's paper. {\sl Biometrika} {\bf 42}, 116--122.


\end{thebibliography}
\end{document}